\documentclass[10pt]{amsart}
\usepackage{amssymb, amsmath, amsthm, amsfonts}
\usepackage{mathrsfs,comment}
\usepackage{graphicx}
\usepackage{todonotes}
\usepackage{dsfont}
\usepackage{euscript}
\usepackage[bookmarks=true]{hyperref}  
\usepackage{url}
\usepackage{dsfont}
\usepackage{enumerate}
\usepackage{tikz}
\usepackage[all]{xy}
\usepackage{tikz-cd} 
\hypersetup{%
  bookmarksnumbered=true,
  colorlinks=true,%
  linkcolor=blue,%
  citecolor=blue,%
  filecolor=blue,%
  menucolor=blue,%
  urlcolor=blue,%
  pdfnewwindow=true,%
  pdfstartview=FitBH}

\newcommand{\bb}[1]{\mathbb{#1}}

\newcommand{\es}[1]{\EuScript{#1}}
\renewcommand{\sf}[1]{\mathsf{#1}}

\DeclareMathOperator{\hocolim}{\mathrm{hocolim}}
\DeclareMathOperator{\holim}{\mathrm{holim}}

\newcommand{\fib}{\mathsf{fib}}

\newcommand{\s}{{\sf{Sp}}}
\DeclareMathOperator{\T}{\sf{Top}_\ast}

\newcommand{\poly}[1]{\mathsf{Poly}^{\leq #1}}
\newcommand{\homog}[1]{\mathsf{Homog}^{#1}}

\newcommand{\Fun}{\sf{Fun}}

\DeclareMathOperator{\nat}{\mathsf{nat}}

\newcommand{\cofrep}{\widehat{c}}

\newcommand{\Ev}{\mathrm{Ev}}

\DeclareMathOperator{\ind}{\mathsf{ind}}
\DeclareMathOperator{\res}{\mathsf{res}}

\newcommand{\op}{\mathrm{op}}

  \newcommand{\adjunction}[4]{
\xymatrix{
#1:#2 \ar@<1ex>[r] &
\ar@<1ex>[l] #3:#4
}}

\newtheorem{thm}{Theorem }[subsection]
\newtheorem{prop}[thm]{Proposition}
\newtheorem{lem}[thm]{Lemma}
\newtheorem{cor}[thm]{Corollary}
\newtheorem{conj}[thm]{Conjecture}
\newtheorem{hypothesis}[thm]{Hypothesis}

\newtheorem*{thm*}{Theorem}

\theoremstyle{definition}
\newtheorem{definition}[thm]{Definition}
\newtheorem{ex}[thm]{Example}
\newtheorem{exs}[thm]{Examples}
\newtheorem{rem}[thm]{Remark}

\begin{document}


\title[The localization of orthogonal calculus]{The localization of orthogonal calculus with respect to homology}

\author{Niall Taggart}

\address{Mathematical Institute, Utrecht University, Budapestlaan 6, 3584 CD Utrecht, The Netherlands}

\email{n.c.taggart@uu.nl}

\date{\today}



\begin{abstract}    
For a set of maps of based spaces $S$ we construct a version of Weiss' orthogonal calculus which depends only on the $S$-local homotopy type of the functor involved. We show that $S$-local homogeneous functors of degree $n$ are equivalent to levelwise $S$-local spectra with an action of the orthogonal group $O(n)$ via a zigzag of Quillen equivalences between appropriate model categories. Our theory specialises to homological localizations and nullifications at a based space. We give a variety of applications including a reformulation of the Telescope Conjecture in terms of our local orthogonal calculus and a calculus version of Postnikov sections. Our results also apply when considering the orthogonal calculus for functors which take values in spectra.
\end{abstract}

\maketitle

\setcounter{tocdepth}{1}
{\hypersetup{linkcolor=black} \tableofcontents}

	\section{Introduction}

	\subsection{Motivation}
Weiss' Orthogonal calculus~\cite{We95} studies functors from the category of real inner product spaces and isometries to the category of based spaces or spectra. The motivation for such a version of functor calculus comes from a desire to study geometric and differential topology through a homotopy theoretic lens. For example, Arone, Lambrechts and Voli{\'c}~\cite{ALV2} and Arone~\cite{Ar09} utilised Weiss' calculus to provide a comprehensive study of the (stable) homotopy type of spaces of embeddings $\mathsf{Emb}(M, N \times \bb{R}^k)$ where $M$ and $N$ are fixed smooth manifolds. More recently Krannich and Randal-Williams~\cite{KRW2021} have studied the Weiss tower of the classifying space $\sf{BTOP}(\bb{R}^k)$ of the group of homeomorphisms of $\bb{R}^k$, to understand the homotopy type of the space of diffeomorphisms of discs. In all of these cases, the authors are only able to ascertain geometric information up to rational homotopy via ad-hoc means. These vastly varying approaches highlight the need for a comprehensive account of the interactions between orthogonal calculus and localizations. 

The theory of localizations at homology theories are ubiquitous and have had wide applications; of particular note is \emph{chromatic homotopy theory} which among other things gives a spectrum level interpretation for the periodic families appearing in the stable homotopy groups of spheres. An extensive amount of effort has been geared toward understanding how localization at homology theories--particularly the chromatic localizations--interact with Goodwillie's calculus of functors~\cite{AroneMahowald,KuhnTate, KuhnTAQ, AroneKuhn}, see e.g.,~\cite{Kuh07} for a survey. Analogous questions remain in Weiss' orthogonal calculus, and we propose a noticeably different approach than those applied to Goodwillie calculus. 

		\subsection*{Overview}
Given a functor $F$ from the category of Euclidean spaces to the category of based spaces or spectra the calculus assigns a tower of functors
\[
\xymatrix{
&&& F \ar@/_1.0pc/[dll] \ar@/_1.0pc/[dl] \ar@/^1.0pc/[dr] \ar@/^1.0pc/[drr] & & \\
\cdots \ar[r] & T_nF \ar[r] & T_{n-1}F \ar[r] & \cdots \ar[r] & T_1F \ar[r] & T_0F
}
\]
called the \emph{Weiss tower} for $F$. The functor $T_nF$ is a categorification of the $n$-th Taylor polynomial from differential calculus. The $n$-th layer of the tower $D_nF$ is the homotopy fibre of the map $T_nF \to T_{n-1}F$, and is a categorification of homogeneous functions from differential calculus. Orthogonal calculus is synonymous with being the most computationally challenging flavour of functor calculus due to the interaction between the highly `geometric' nature of the objects of study and the highly homotopical constructions.

Let $\es{C}$ denote either the category of based spaces or spectra. Given a set $S$ of maps in $\es{C}$ we produce an \emph{$S$-local Weiss tower} of the following form:
\[
\xymatrix{
&&& F \ar@/_1.0pc/[dll] \ar@/_1.0pc/[dl] \ar@/^1.0pc/[dr] \ar@/^1.0pc/[drr] & & \\
\cdots \ar[r] & T_n^SF \ar[r] & T_{n-1}^SF \ar[r] & \cdots \ar[r] & T_1^SF \ar[r] & T_0^SF \\
& D_n^S F \ar[u] & D_{n-1}^S F \ar[u] & & D_1^S F \ar[u] & \\
}
\]

To understand the $S$-local Weiss tower, we utilise Bousfield's~\cite{BousfieldLocalSpectra,BousfieldLocalSpaces} interpretation of localizations in terms of model structures on $\es{C}$. We begin by constructing a model structure, denoted $\poly{n}(\es{J}_0, L_S\es{C})$, which captures the homotopy theory of functors which are $S$-locally polynomial of degree less than or equal $n$. Under some assumptions on the set of localizing objects the composite $T_nL_S$ is a fibrant replacement functor hence satisfying the necessary universal property.

We further construct a model structure, denoted $\homog{n}(\es{J}_0, L_S\es{C})$, which captures the homotopy theory of functors which are $S$-locally homogeneous of degree $n$. Through a zigzag of Quillen equivalences we characterise the $S$-local $n$-homogeneous functors in terms of appropriately $S$-local spectra with an action of $O(n)$.

\begin{thm*}[Corollary \ref{cor: zigzag}]
Let $S$ be a set of maps of based spaces and $n \geq 0$. There is a zigzag of Quillen equivalences
\[
\homog{n}(\es{J}_0, L_S\T) \simeq_Q {\s}(L_S\T)[O(n)],
\]
where ${\s}(L_S\T)[O(n)]$ is the category of levelwise $S$-local spectra with an action of $O(n)$.
\end{thm*}

\begin{thm*}[Corollary \ref{cor: zigzag spectra}]
Let $S$ be a set of maps of spectra and $n \geq 0$. There is a zigzag of Quillen equivalences
\[
\homog{n}(\es{J}_0, L_S\s) \simeq_Q L_S{\s}[O(n)],
\]
where $L_S{\s}[O(n)]$ is the category of $S$-local spectra with an action of $O(n)$.
\end{thm*}

In particular, an $S$-local $n$-homogeneous functor $F$ is determined by and determines an appropriately $S$-local spectrum with an $O(n)$-action, denoted $\partial_n^S F$. On the derived level, we obtain a computationally accessible classification theorem for $S$-local homogeneous of degree $n$ functors.

\begin{thm*}[Theorem \ref{thm: classification of E-local n-homog}]\label{thm: classification of spectral E-local n-homog}
Let $S$ be a set of maps of in $\es{C}$ and $n \geq 1$. 
\begin{enumerate}
\item A $\T$-valued $S$-local $n$-homogeneous functor $F$ is objectwise weakly equivalent to the functor 
\[
V \longmapsto \Omega^\infty[ (S^{\bb{R}^n \otimes V} \wedge \partial_n^S F)_{hO(n)}],
\]
and any functor of the above form is objectwise $S$-local and $n$-homogeneous. 
\item A $\s$-valued $S$-local $n$-homogeneous functor $F$ is objectwise weakly equivalent to the functor
\[
V \longmapsto (S^{\bb{R}^n \otimes V} \wedge \partial_n^S F)_{hO(n)},
\]
and any functor of the above form is objectwise $S$-local and $n$-homogeneous. 
\end{enumerate}
\end{thm*}

		\subsection*{Applications}
We envision that the applications of this local version of orthogonal calculus are vast. For example, extending the rational computations of \cite{ALV2, Ar09, KRW2021} to higher chromatic height or another perspective on the full understanding of the Weiss tower of $\sf{BO}(-)$ in $v_n$-periodic homotopy theory achieved by Arone \cite{Ar02} using computations of Arone and Mahowald \cite{AroneMahowald}. 

Very little of our results use the fact that the target category is based spaces or spectra. The largest hurdle to having a theory of localizations of orthogonal calculus with target any (simplicial cofibrantly generated) model category is the development of orthogonal calculus in this realm. We hope that our expos\'{e} of orthogonal calculus with target space a localization of spaces or spectra will motivate the construction of orthogonal calculus based on more general homotopy theories such as arbitrary model categories or $\infty$-categories.

In the last part of this paper, we give several initial applications, of which we survey here.

\subsubsection*{Bousfield classes}
In \cite{BousfieldLocalSpectra}, Bousfield introduced an equivalence relation on the stable homotopy category that has turned out to be of extreme importance. Define Bousfield class $\langle E \rangle$ of a spectrum $E$ to be the collection of $E$-acyclic spectra, and say that $E$ and $E'$ are Bousfield equivalent if and only if $\langle E \rangle = \langle E' \rangle$. These Bousfield classes assemble into a lattice, the understanding of which has been a major task in stable homotopy theory. For example, the Nilpotence Theorem of Devanitz, Hopkins, and Smith \cite{NilpotenceI, NilpotenceII} is equivalent to a classification of the Bousfield classes for \emph{finite spectra}. The Bousfield lattice has many interesting interactions with homological localizations of orthogonal calculus. 

\begin{thm*}[{Example \ref{ex: BF classes}}]
Let $E$ and $E'$ be spectra. The $E$-local orthogonal calculus is equivalent to the $E'$-local orthogonal calculus if and only if $E$ and $E'$ are Bousfield equivalent.
\end{thm*}

Fix a prime $p$. Ravenel's height $n$ Telescope Conjecture~\cite[Conjecture 10.5]{RavenelLocalizations} is the statement that the height $n$ Morava $K$-theory, $K(n)$, is Bousfield equivalent to $T(n)$, the telescope of any $v_n$-self map on a finite type $n$ complex. The Telescope Conjecture is trivial at height $n=0$, has been verified at height $n=1$ and at all primes by Bousfield~\cite{BousfieldLocalSpectra}, Mahowald~\cite{MahowaldboResolutions} and Miller~\cite{MillerASS}, but in general, is widely believed to be false.

\begin{thm*}[Corollary {\ref{cor: telescope implies agree towers}}]
The height $n$ Telescope Conjecture holds if and only if the $K(n)$-local orthogonal calculus and the $T(n)$-local orthogonal calculus are equivalent. 
\end{thm*}

The Weiss tower of a functor $F$ produces a spectral sequence as it is a tower of fibrations. We call this spectral sequence the \emph{Weiss spectral sequence}. From a computational perspective we obtain the following relation between the Telescope Conjecture and the local Weiss spectral sequences.

\begin{thm*}[Lemma \ref{lem: telescope and WSS}]
 If the height $n$ Telescope Conjecture holds, then for all $r \geq 0$, the $r$-th page of the $T(n)$-local Weiss spectral sequence is isomorphic to the $r$-th page of the $K(n)$-local Weiss spectral sequence. 
\end{thm*}

\subsubsection*{Nullifications}
For functors from the category of Euclidean spaces to the category of based spaces we also consider localization at a based space $W$, which is sometimes referred to as \emph{nullification}. In this setting $W$-local objects are also called $W$-periodic, following Bousfield \cite{BousfieldPeriodicity} and Farjoun \cite{FarjounCellular}.

We give alternative constructions for the $n$-polynomial and $n$-homogenous model structures when the localization is a nullification. These alternative constructions yield an identical $n$-polynomial model structure but sheds new light on some of the formal properties of the model structure, and yield an $n$-homogeneous model structure which is Quillen equivalent to the original $W$-local model structure via the identity functor. These alternative descriptions are particularly useful when considering Postnikov sections of orthogonal calculus.

The results obtained for nullifications do not hold for more general localizations as the techniques employed rely crucially on  a right properness condition on the model categories. We show in Proposition \ref{prop: local model as a nullification} that the right proper condition is satisfied if and only if the localization is a nullification. This is an extension of a remark of Bousfield in \cite{Bo01}.

\subsubsection*{Postnikov Sections}

Considering nullifications with respect to the spheres produces a theory of Postnikov sections in orthogonal calculus. We prove that our $S^{k+1}$-local projective model structure on the category of functors from Euclidean spaces to based spaces is identical to the model structure of $k$-types in the category of functors from Euclidean spaces to based spaces in the sense of $k$-types in an arbitrary model category developed by Guti\'{e}rrez and Roitzheim~\cite[\S4]{GR17}.

\begin{thm*}[Proposition {\ref{prop: GR k-types in orthogonal functors}}]
Let $k \geq 0$. The model structure of $k$-types in orthogonal functors is identical to the $S^{k+1}$-local model structure, that is, there is an equality of model structures,
\[
P_k{\Fun}(\es{J}_0, \T) := L_{W_k}{\Fun}(\es{J}_0, \T) = {\Fun}(\es{J}_0, L_{S^{k+1}}\T). 
\]	
\end{thm*}

As an application we produce a tower of model categories
\[
\cdots \longrightarrow \homog{n}(\es{J}_0, P_k\T) \longrightarrow \cdots \longrightarrow \homog{n}(\es{J}_0, P_0\T),
\]
where $P_k\T$ denotes the $S^{k+1}$-local model structure on based spaces. By applying the theory of homotopy limits of model categories, we show that the $n$-homogeneous model structure of Barnes and Oman \cite[Proposition 6.9]{BO13} is the homotopy limit of this tower, in the following sense.

\begin{thm*}[{Corollary \ref{cor: homog holim}}]
There is a Quillen equivalence
\[ 
\homog{n}(\es{J}_0, \T) \simeq_{Q} \underset{k}{\holim}~ \homog{n}(\es{J}_0, P_k\T).
\]
\end{thm*}

	\subsection*{Relation to other work}

This work is intimately related to the rational orthogonal calculus developed by Barnes~\cite{Ba17},  by replacing our generalised homology theory $E_\ast$ with rational homology one recovers Barnes' theory.

Unstable chromatic homotopy theory can be described algebraically, via Heuts' \cite{HeutsAlgModels} \emph{algebraic model} for $v_n$-periodic spaces via an equivalence (of $\infty$-categories) with Lie algebras in $T(n)$-local spectra. This model indicated that there is likely a relationship between $v_n$-periodic orthogonal calculus and orthogonal calculus of Heuts' Lie algebra models. Such an equivalence at chromatic height zero suggests a relationship between rational orthogonal calculus and the algebraic models for rational homotopy theory of Sullivan and Quillen \cite{QuillenRational, SullivanRational}. This together with Barnes' \cite{Ba17} model for rational $n$-homogeneous functors using the classification of rational spectra with an $O(n)$-action as torsion modules over the rational cohomology ring of ${\sf BSO}(n)$ of Greenlees and Shipley \cite{GS14} suggests the existence of \emph{algebraic model calculi}. We plan to return to this in future work. 

This work also forms part of an extensive program to go ``beyond orthogonal calculus" which was initiated in the Ph.D. thesis of the author \cite{Taggartthesis}, together with a series of articles exploring extensions of the orthogonal calculus and the relations between these, \cite{TaggartUnitary, TaggartOCandUC, TaggartReality, TaggartRealityUnitary}. This extensive project hopes to illuminate our understanding of orthogonal calculus which (at least relative to Goodwillie calculus) remains largely unexplored.

The future applications of the homological localization of orthogonal calculus are abounding. For example in the recent work of Beaudry, Bobkova, Pham, and Xu~\cite{BeaudryBobkovaPhamXu}, the authors compute the $tmf$-homology of $\bb{R}P^2$, where $tmf$ denotes the connective spectrum of topological modular forms. Their computation for $\bb{R}P^2$ and the $tmf$-local Weiss tower for the functor $V \mapsto \bb{R}P(V)$ should yield a calculation of the $tmf$-homology of $\bb{R}P^k$ for all $k$. Such a connection would, for example, feed into a chromatic understanding of block structures, see e.g., \cite{Ma07}.

		\subsection*{Conventions}

We work extensively with model categories and refer the reader to the survey article \cite{DS95} and the textbooks \cite{Ho99, Hi03} for a detailed account of the theory. We further assume the reader has familiarity with orthogonal calculus, references for which include \cite{BO13, We95}.

The category $\T$ will always denote the category of based compactly generated weak Hausdorff spaces, and we will, for brevity, call the objects of this category ``based spaces". The category of based spaces will always be equipped with the Quillen model structure unless specified otherwise. The weak equivalences are the weak homotopy equivalences and fibrations are Serre fibrations. This is a cellular, proper and topological model category with sets of generating cofibrations and acyclic cofibrations denoted by $I$ and $J$, respectively. 

Unless otherwise stated the word ``spectra'' is synonymous with the phrase ``orthogonal spectra'', details of which can be found in \cite{MMSS01} in the non-equivariant case, and \cite{MM02} in the equivariant situation. 

We will denote by $\es{C}$ either the category of based spaces or of orthogonal spectra.

		\subsection*{Acknowledgements}

This work has benefited from helpful conversations and comments from D.~Barnes, T.~Barthel, G.~Heuts, I.~Moerdijk and J.~Williamson. We are particularly grateful to S.~Balchin for reading an earlier version of this material. We extend our thanks to the meticulous referee who has greatly enhanced this article by taking (in their own words) ``a long time'' to check the numerous technical results. We also thank the Max Plank Institute for Mathematics for its hospitality during part of the writing process.

\part{Local orthogonal calculus}

	\section{Orthogonal functors}

Denote by $\es{C}$ the category $\T$ of based topological spaces or the category $\s$ of (orthogonal) spectra. Define $\es{J}$ to be the category with objects finite-dimensional inner product subspaces of $\bb{R}^\infty$ and with morphisms the linear isometries. Define $\es{J}_0$ to be the category with the same objects and $\es{J}_0(U,V) = \es{J}(U,V)_+$. The morphism set $\es{J}(U,V)$ may be topologised as the Stiefel manifold of $\dim(U)$-frames in $V$. As such, $\es{J}$ is a topologically enriched category, and $\es{J}_0$ is enriched in based spaces. Since the functor 
\[
\Sigma^\infty \colon \T \longrightarrow \s,
\]
is symmetric monoidal, see e.g., \cite[Lemma II.4.8]{MM02}, we may enhance the topological enrichment of $\es{J}_0$ to a spectral enrichment, resulting in a category $\es{J}_0^\s$, whose class of objects agrees with the class of objects in $\es{J}_0$, and morphism spectrum 
\[
\es{J}_0^\s(V,W) = \Sigma^\infty \es{J}_0(V,W).
\]
We will omit the superscript ``$\s$'' when confusion is unlikely to occur.

The category $\Fun(\es{J}_0, \es{C})$ of $\es{C}$-enriched functors from $\es{J}_0$ to $\es{C}$ is the category of input functors for orthogonal calculus. We will refer to such functors as \emph{$\es{C}$-valued orthogonal functors} or simply \emph{orthogonal functors} when confusion is unlikely. Examples of orthogonal functors are abound in geometry, topology and homotopy theory, and examples of $\T$-valued orthogonal functors include:
\begin{enumerate}
	\item the one-point compactification functor $\bb{S} \colon V \mapsto S^V$;
	\item the functor ${\sf{BO}}(-) \colon V \mapsto {\sf{BO}}(V)$ which sends an inner product space to the classifying space of its orthogonal group;
	\item the functor ${\sf{BTOP}}(-)\colon V \mapsto {\sf{BTOP}}(V)$, which sends an inner product space $V$ to ${\sf{BTOP}}(V)$, the classifying space of the space of self-homeomorphisms of $V$;
	\item the functor ${\sf{BDiff}^b}(M \times -) \colon V \mapsto {\sf{BDiff}^b}(M \times V)$, which for a fixed smooth and compact manifold $M$ sends an inner product space $V$ to the classifying space of the group of bounded diffeomorphisms from $M \times V$ to $M \times V$ which are the identity on $\partial M \times V$; and,
	\item the restriction of an endofunctor on based spaces to evaluation on spheres\footnote{Endofunctors of based spaces are particularly interesting from a homotopy theoretic point of view when you restrict to the values on spheres, see e.g., \cite{BehrensEHP, AroneMahowald,Ar02}. In particular for $F$ the identity functor, the Weiss tower of $F \circ \bb{S} = \bb{S}$ and the Goodwillie tower for $F$ agree up to weak equivalence \cite{BE16}, hence orthogonal calculus is intimately related to understanding the (stable) homotopy groups of spheres.}.
\end{enumerate}

The category of orthogonal functors may be equipped with a projective model structure. 

\begin{prop}
There is a model category structure on the category of orthogonal functors $\Fun(\es{J}_0, \es{C})$ with weak equivalences and fibrations defined objectwise. This model structure is cellular, proper and topological, and in the case of $\s$-valued orthogonal functors, this model structure is spectral and stable.
\end{prop}

		\subsection{Local input functors}

The `base' model structure for the $S$-local orthogonal calculus will be the $S$-local model structure on the category of orthogonal functors.
 
 \begin{prop}\label{prop: S-local proj model structure}
Let $S$ be a set of maps in $\es{C}$. There is model structure on the category of orthogonal functors such that a map is a weak equivalence or fibration if it is an objectwise $S$-local equivalence or a objectwise $S$-local fibration in $\es{C}$, respectively.  This model structure is cellular, left proper and topological, and in the case of $\s$-valued orthogonal functors this model structure is spectral. We call this model structure the $S$-local projective model structure and denote it by ${\Fun}(\es{J}_0, L_S\es{C})$. 
\end{prop}
\begin{proof}
This model structure is an instance of a projective model structure on a category of functors, see e.g., \cite[Theorem 11.6.1]{Hi03}.
\end{proof}

\begin{ex}
For $E_\ast$ a generalised homology theory, the model structure of Proposition \ref{prop: S-local proj model structure} has weak equivalences the objectwise $E_\ast$-isomorphisms, and fibrant objects objectwise $E_\ast$-local objects. This follows since the $E_\ast$-localization of spaces and spectra exist by work of Bousfield, see e.g., \cite{BousfieldLocalSpaces, BousfieldLocalSpectra}.
\end{ex}

	\section{Polynomial functors}

		\subsection{Polynomial functors}

Polynomial functors behave in many ways like polynomial functions from classical calculus, e.g., a functor which is polynomial of degree less than or equal $n$, is polynomial of degree less than or equal $n+1$. We give only the necessary details here and refer the reader to \cite{We95} or \cite{BO13} for more details on polynomial functors in orthogonal calculus.

\begin{definition}
An orthogonal functor $F$ is \emph{polynomial of degree less than or equal $n$} if $F$ is objectwise fibrant and for each $U \in \es{J}_0$, the canonical map 
\[
F(U) \longrightarrow \underset{0 \neq V \subseteq \bb{R}^{n+1}}{\holim} F(U \oplus V) =: \tau_n F(U),
\]
is a weak homotopy equivalence. Functors which are polynomial of degree less than or equal $n$ will sometimes be referred to as \emph{$n$--polynomial} functors.
\end{definition}

\begin{rem}
Given an orthogonal functor $F$ and an inner product space $U$ we can restrict the orthogonal functor $F(U \oplus -)$ to a functor 
\[
F(U \oplus -) \colon \es{P}(\bb{R}^{n+1}) \longrightarrow \T,
\]
where $\es{P}(\bb{R}^{n+1})$ is the poset of finite-dimensional inner product subspaces of $\bb{R}^{n+1}$. Such functors are deserving of the name \emph{$\bb{R}^{n+1}$-cubes} by analogy with cubical homotopy theory. The orthogonal functor $F$ being $n$-polynomial is equivalent to asking that for each $U$ this restricted functor is homotopy cartesian. Informally speaking, orthogonal calculus can be thought of a calculus built from \emph{$\bb{R}^n$-cubical homotopy theory} in a similar way to how Goodwillie calculus is built from cubical homotopy theory.
\end{rem}

There is a functorial assignment of a universal (up to homotopy) $n$-polynomial functor to any orthogonal functor $F$. It is the \emph{$n$-polynomial approximation} of $F$, and is defined as 
\[
T_n F(U) = \hocolim( F(U) \longrightarrow \tau_n F(U) \longrightarrow \cdots \longrightarrow \tau_n^kF(U) \longrightarrow \cdots).
\]

In \cite[Proposition 6.5 $\&$ Proposition 6.6]{BO13}, Barnes and Oman construct a localization of the projective model structure on the category of orthogonal functors which captures the homotopy theory of $n$-polynomial functors, in particular the $n$-polynomial approximation functor is a fibrant replacement. There are two equivalent ways to consider this model structure; as the Bousfield-Friedlander localization of $\Fun(\es{J}_0, \es{C})$ at the $n$-polynomial approximation endofunctor 
\[
T_n \colon \Fun(\es{J}_0, \es{C}) \longrightarrow \Fun(\es{J}_0, \es{C}),
\]
 or, as the left Bousfield localization at the set
 \[
\es{S}_n = \{ S\gamma_{n+1}(U,V)_+ \longrightarrow \es{J}_0(U,V) \ \mid \ U, V \in \es{J}_0\},
\] 
for $\T$-valued orthogonal functors, or the set $\Sigma^\infty\es{S}_n = \{\Sigma^\infty f \mid f \in \es{S}_n\}$ for $\s$-valued orthogonal functors, where $S\gamma_{n+1}(V,W)$ is the sphere bundle of the $(n+1)$-fold Whitney sum of the orthogonal complement bundle over the space of linear isometries $\es{J}(V,W)$. 

\begin{prop}[{\cite[Proposition 6.5]{BO13}}]
There is a model category structure on the category of orthogonal functors with weak equivalences the $T_n$-equivalences\footnote{A map $f \colon X \to Y$ is a $T_n$-equivalence if $T_n(f) \colon T_nX \to T_nY$ is a objectwise weak equivalence.} and fibrations those objectwise fibrations $f \colon X \to Y$ such the square 
\[
\xymatrix{
X \ar[r] \ar[d] & T_nX \ar[d] \\
Y \ar[r] & T_nY \\
}
\]
is a homotopy pullback in the projective model structure. This model structure is cellular, proper and topological, and in the case of $\s$-valued orthogonal functors this model structure is spectral. We call this the $n$-polynomial model structure and denote it by $\poly{n}(\es{J}_0, \es{C})$. 
\end{prop}

		\subsection{Local polynomial functors}

The definition of an $S$-locally $n$-polynomial functor is the analogous definition of an $n$-polynomial functor when the base model category is $L_S\es{C}$, i.e., a objectwise fibrant functor which satisfies a cartesian $\bb{R}^{n+1}$-cube condition. 

\begin{definition}
Let $S$ be a set of maps in $\es{C}$. An orthogonal functor is \emph{$S$-locally $n$-polynomial} if it is objectwise $S$-local and $n$-polynomial.
\end{definition}

The $S$-locally $n$-polynomial model structure is an iterated left Bousfield localization involving the set $\es{S}_n$ and the set 
\[
J_S = \{ \es{J}_0(U,-) \wedge j \mid U \in \es{J},, j \in J_{L_S\es{C}}\},
\]
as this iterative localization will necessarily have the $S$-locally $n$-polynomial functors as fibrant objects. This model structure was first constructed by Barnes \cite{Ba17} for the rationalization of $\T$-valued orthogonal functors.

\begin{prop}\label{prop: S-local n-poly model structure}
Let $S$ be a set of maps in $\es{C}$. There is model category structure on the category of orthogonal functors with cofibrations the projective cofibrations, and fibrant objects the $S$-locally $n$-polynomial functors. This model structure is cellular, left proper, topological, and in the case of $\s$-valued orthogonal functors this model structure is spectral. We call this model structure the $S$-local $n$-polynomial model structure and denote it by $\poly{n}(\es{J}_0, L_S\es{C})$. 
\end{prop}
\begin{proof}
The process of left Bousfield localizations may be iterated and it follows that the $J_S$-localization of the $n$-polynomial model structure and the $\es{S}_n$-localization of the $S$-local projective model structure are identical, and have as cofibrations the projective cofibrations.

For the fibrant objects, notice that the model structure is equivalently described as the left Bousfield localization of the projective model structure with respect to the set of maps $\es{S}_n \cup J_S$. By definition an object $X$ is $\es{S}_n \cup J_S$-local if and only if it is both $\es{S}_n$-local and $J_S$-local, and hence the fibrant objects are precise those $S$-locally $n$-polynomial functors.
\end{proof}

The $S$-local $n$-polynomial model structure behaves precisely like a left Bousfield localization of the $n$-polynomial model structure in the following sense. 

\begin{lem}\label{lem: QA for polynomials}
Let $S$ be a set of maps in $\es{C}$. The adjoint pair
\[
\adjunction{ \mathds{1}}{\poly{n}(\es{J}_0, \es{C})}{\poly{n}(\es{J}_0, L_S\es{C})}{ \mathds{1}},
\]
is a Quillen adjunction. 
\end{lem}
\begin{proof}
The left adjoint preserves cofibrations since the classes of cofibrations are identical. The right adjoint is right Quillen since it preserves fibrant objects as every $S$-locally $n$-polynomial functor is necessarily $n$-polynomial. 
\end{proof}

The composite $T_nL_S$ need not be a fibrant replacement functor in the $S$-local $n$-polynomial model structure since the class of $S$-local objects need not be closed under filtered homotopy colimits. Imposing a condition on the set $S$ which forces $T_nL_S$ to be $S$-local in turn forces $T_nL_S$ to be a functorial fibrant replacement.

\begin{prop}
Let $S$ be a set of maps in $\es{C}$. If the class of $S$-local objects is closed under sequential homotopy colimits, then the weak equivalences of the $S$-local $n$-polynomial model structure are those maps $f \colon X \to Y$ 	such that the induced map
\[
T_nL_S f \colon T_n L_SX \longrightarrow T_nL_SY,
\]
is a $S$-local equivalence. In particular, The composite $T_nL_S$ is a functorial fibrant replacement in the $S$-local $n$-polynomial model structure.
\end{prop}
\begin{proof}
We apply \cite[Lemma 5.5]{Ba17} which shows that a map $f \colon X \to Y$ is weak equivalence in the iterated left Bousfield localization if and only if 
\[
L_S f \colon L_S X \longrightarrow L_S Y
\] 
is a $\es{S}_n$-local equivalence. This last is equivalent to $L_S f \colon L_S X \to L_S Y$ being a $T_n$-equivalence, i.e., $T_nL_S f \colon T_nL_S X \to T_nL_S Y$ being a objectwise weak equivalence. Since both domain and codomain of this map are $S$-local checking this map is a objectwise weak equivalence is equivalent to checking that it is an $S$-local equivalence by the $S$-local Whitehead's Theorem.
\end{proof}

\begin{rem}
Let $S$ be a set of maps in $\es{C}$. To ease notation, we will denote the composite $T_nL_S$ by $T_n^S$. In particular, for $E$ a spectrum we denote the composite functor $T_nL_E$ by $T_n^E$. In general, $T_n^S$ need not be $S$-local, but will be when the class of $S$-local objects is closed under sequential homotopy colimits.
\end{rem}

\begin{exs}\hspace{10cm}
\begin{enumerate}
	\item For a finite cell complex $W$, $T_n^WF$ is $W$-local (or $W$-periodic) for all $\T$-valued orthogonal functors $F$.
	\item For localization at the Eilenberg-Maclane spectrum associated to a subring $R$ of the rationals, $T_n^{HR}F$ is $HR$-local for all orthogonal functors $F$. 
	\item For $E$ a spectrum such that the associated localization of spectra is smashing, $T_n^EF$ is $E$-local for all $\s$-valued orthogonal functors $F$. 
\end{enumerate}
\end{exs}

	\section{Differentiation}

The analogy between orthogonal calculus and differential calculus (Taylor's version) indicated the existence of an inductive `formula' for the $n$-polynomial approximation. The building blocks of such a `formula' are the derivatives of the functor under consideration. 

		\subsection{The derivatives}\label{ssection: derivatives}

The orthogonal complement of the pullback of the tautological bundle to the Stiefel manifold $\es{J}_0(V,W)$ is a vector bundle $\gamma_1(V,W)$ with fibre over an isometry $f$ given by $f(V)^\perp$. For $n \geq 0$, we denote the $n$-fold Whitney sum of $\gamma_1(V,W)$ by $\gamma_n(V,W)$. Define $\es{J}_n$ to be the category with the same objects as $\es{J}$ and morphism space $\es{J}_n(U,V)$ given as the Thom space of $\gamma_n(U,V)$. Define $\es{J}_n^\s$ to be the spectral enriched version of $\es{J}_n$, i.e., the category with the same objects but morphism spectrum given by $\es{J}^\s(V,W) = \Sigma^\infty \es{J}_n(V,W)$. The standard action of $O(n)$ on $\bb{R}^n$ via the regular representation induces an action on the vector bundles that is compatible with the composition, hence $\es{J}_n$ is naturally enriched over based spaces with an $O(n)$-action.

Recall that $\es{C}$ denotes the category of based spaces or spectra. We denote by $\es{C}[O(n)]$ the category of $O(n)$-objects in $\es{C}$. For $\es{C}=\T$, this recovers the category of $O(n)$-spaces, and for $\es{C}=\s$, this is the category of spectra with an $O(n)$-action. Let $0 \leq m \leq n$. The inclusion $i_m^n \colon \bb{R}^m \to \bb{R}^n$ induces a functor $i_m^n \colon \es{J}_m \to \es{J}_n$. Postcomposition with $i_m^n$ induced a topological functor 
\[
\res_m^n \colon \Fun(\es{J}_n, \es{C}) \longrightarrow \Fun(\es{J}_m, \es{C}),
\]
which by \cite[Proposition 2.1]{We95} has a right adjoint 
\[
\ind_m^n \colon \Fun(\es{J}_m, \es{C}) \longrightarrow \Fun(\es{J}_n, \es{C}),
\]
the right Kan extension along $i_m^n$, and is given by
\[
\ind_m^n F(U) = \nat_m(\es{J}_n(U,-), F),
\]
where $\nat_m(-,-)$ denotes the space of natural transformations in $\Fun(\es{J}_m, \es{C})$ and $\es{J}_n(U,-)$ is considered as an object of $\Fun(\es{J}_m, \es{C})$ by restriction. Combining the restriction and induction functors with change of group adjunctions from \cite{MM02}, we obtain an adjoint pair
\[
\adjunction{\res_m^n/ O(n-m)}{ {\Fun}_{O(n)}(\es{J}_n, \es{C}[O(n)])}{ {\Fun}_{O(m)}(\es{J}_m, \es{C}[O(m)])}{{\ind}_m^n \sf{CI}},
\]
see \cite[\S 4]{BO13}, where ${\Fun}_{O(n)}(\es{J}_n, \es{C}[O(n)])$ is the category of $\es{C}[O(n)]$-enriched functors from $\es{J}_n$ to $\es{C}[O(n)]$. We refer to this category as the \emph{$n$-th intermediate category} on the point of its role as an intermediate in the classification of $n$-homogeneous functors, see Section \ref{section: classification of $n$-homogeneous functors}.

\begin{definition}\label{def: derivative}
Let $F$ be an orthogonal functor. For $n \geq 0$, the \emph{$n$-th derivative of $F$} is given by $\ind_0^n {\sf{CI}} F$. In which case, we write $\ind_0^n \varepsilon^*F$ or $F^{(n)}$. 
\end{definition}

Restricted evaluation in the $n$-th intermediate category induces structure maps of the form
\[
X(V) \wedge S^{\bb{R}^n \otimes W} \longrightarrow X(V \oplus W), 
\]
for $X \in {\Fun}_{O(n)}(\es{J}_n, \es{C}[O(n)])$ and $V,W \in \es{J}_n$, see e.g., \cite[\S7]{BO13}. It is thus reasonable to think of the objects of the $n$-th intermediate category as \emph{spectra of multiplicity $n$}, see e.g., \cite[\S9]{We95}. This idea leads to an object $Z$ in the $n$-th intermediate category being called an \emph{$n\Omega$-spectrum} if the adjoint structure maps
\[
Z(V) \longrightarrow \Omega^{\bb{R}^n \otimes W} Z(V \oplus W),
\]
are weak equivalences in $\es{C}$, and a map $f : X \to Y$ in the $n$-th intermediate category being called an $n$-stable equivalence if the induced map 
\[
f^* \colon [Y, Z] \longrightarrow [X,Z]
\]
on objectwise homotopy classes of maps is an isomorphism for all $n\Omega$-spectra $Z$. With these definitions we get an $n$-stable model structure on the $n$-th intermediate category analogous to the stable model structure on spectra, see e.g., \cite[\S7]{BO13}.

\begin{prop}[{\cite[Proposition 7.14]{BO13}}]
There is a model category structure on the $n$-th intermediate category with weak equivalences the $n$-stable equivalences and fibrations the objectwise fibrations $X \to Y$ such that the square 
\[
\xymatrix{
X(U) \ar[r] \ar[d] & \Omega^{\bb{R}^n \otimes V}X(U \oplus V) \ar[d] \\
Y(U) \ar[r] & \Omega^{\bb{R}^n \otimes V}Y(U \oplus V)
}
\]
is a homotopy pullback in $\es{C}$ for all $U,V \in \es{J}_n$. The fibrant objects are the $n\Omega$-spectra. This model structure is cellular, proper, stable and topological, and in the case of $\s$-valued orthogonal functors, this model structure is spectral. We call this the $n$-stable model structure and denote it by ${\Fun}_{O(n)}(\es{J}_n, \es{C}[O(n)])$.
\end{prop}

		\subsection{The local $n$-stable model structure}

We now equip the $n$-th intermediate category with an $S$-local model structure which will be intermediate in our classification of $S$-local $n$-homogeneous functors as appropriately\footnote{Here ``appropriately'' means levelwise $S$-local spectra for $\T$-valued orthogonal functors and $S$-local spectra for $\s$-valued orthogonal functors.} $S$-local spectra with an action of $O(n)$. This model structure was first defined by Barnes \cite{Ba17} for the rationalization of $\T$-valued orthogonal functors. 

\begin{prop}\label{prop: E-local intermediate}
Let $S$ be a set of maps in $\es{C}$. There is a model category structure on the $n$-th intermediate category with cofibrations the cofibrations of the $n$-stable model structure and fibrant objects the $n\Omega$-spectra which are objectwise $S$-local. This model structure is cellular, left proper and topological, and in the case of $\s$-valued orthogonal functors, this model structure is spectral. We call this the $S$-local $n$-stable model structure and denote it by $L_S{\Fun}_{O(n)}(\es{J}_n, \es{C}[O(n)])$.
\end{prop}
\begin{proof}
This model structure is the left  Bousfield localization of the $n$-stable model structure at the set 
\[
\es{Q}_n = \{ O(n)_+ \wedge \es{J}_n(U, -) \wedge j \mid U \in \es{J}, j \in J_{L_S\es{C}}\}. \qedhere
\]
\end{proof}

We record the following fact which will prove useful later. 

\begin{lem}\label{lem: derivative E-local}
Let $S$ be a set of maps in $\es{C}$. If $F$ is an $S$-local functor, then $F^{(n)}=\ind_0^nF$ is $S$-local.
\end{lem}
\begin{proof}
The objectwise smash product
\[
(-) \wedge (-) \colon {\Fun}(\es{J}_n, L_S\es{C}) \times L_S\es{C} \longrightarrow {\Fun}(\es{J}_n, L_S\es{C}),
\]
is a Quillen bifunctor, and the result follows from the definition of $\ind_0^nF$.  
\end{proof}


		\subsection{The derivatives as spectra}

The $n$-th derivative ($n \geq 0$) is naturally an object of the $n$-th intermediate category, i.e., is a spectrum of multiplicity $n$. This multiplicity may be reduced to $n=1$ through a Quillen equivalence
\[
\adjunction{(\alpha_n)_!}{{\Fun}_{O(n)}(\es{J}_n, O(n)\T)}{{\s}[O(n)]}{(\alpha_n)^*},
\]
in the topological case see e.g., \cite[\S8]{BO13}, and by a series of Quillen equivalences
\[
\xymatrix@C+1ex{
\Fun(\es{J}_n, \s[O(n)]) \ar@<1ex>[r]^{(\alpha_n)_!} & \s(\s[O(n)]) \ar@<-1ex>[r]_{\rm{Ev}_0} \ar@<1ex>[l]^{(\alpha_n)^\ast} &  \s[O(n)]  \ar@<-1ex>[l]_{F_0}
}
\]
in the spectral case, see e.g., \cite[\S11]{BO13}. Here $\s(\s[O(n)])$ denotes the category of spectrum objects in spectra with an $O(n)$-action or equivalently, orthogonal bispectra with an $O(n)$-action, and is Quillen equivalent to orthogonal spectra by arguments similar to \cite[Theorem 5.1]{HoveySpectra} or \cite[Theorem 3.8.2]{SS03}. 

\begin{ex}
The (spectrum representing the) $n$-th derivative of the $\T$-valued orthogonal functor $\sf{BO}(-)$ have been completely calculated by Arone \cite{Ar02}. Weiss \cite{We95} calculated the first few examples by hand, for instance the first derivative is the sphere spectrum with trivial $O(1)$-action, the second derivative is the shifted sphere spectrum $\bb{S}^{-1}$ with trivial action, and the third derivative is the $2$-fold loops on the mod-$3$ Moore spectrum $\Omega^2 (\bb{S}/3)$. Higher derivatives have a striking resemblance with the  Goodwillie derivatives of the identity functors on based spaces.
\end{ex}

We now prove that this result holds $S$-locally for any set $S$ of maps in our category $\es{C}$. Since the adjunctions are slightly different, we prove each separately. 

\begin{thm}\label{thm: E-local spectra and intermediate}
Let $S$ be a set of maps of based spaces. The adjoint pair
\[
\adjunction{(\alpha_n)_!}{L_S{\Fun}_{O(n)}(\es{J}_n, O(n)\T)}{{\s}(L_S\T)[O(n)]}{(\alpha_n)^*},
\]
is a Quillen equivalence between the $S$-local model structures. 
\end{thm}
\begin{proof}
For the Quillen adjunction apply \cite[Theorem 3.3.20(1)]{Hi03}, noting that there is an isomorphism 
\[
(\alpha_n)_!(O(n)_+ \wedge \es{J}_n(U, -) \wedge j) \cong O(n)_+ \wedge \es{J}_1(\bb{R}^n \otimes U, -) \wedge j,
\]
for $j$ a generating acyclic cofibration for the $S$-local model structure on based spaces. 

By \cite[Proposition 8.3]{BO13} the adjoint pair
\[
\adjunction{(\alpha_n)_!}{{\Fun}_{O(n)}(\es{J}_n, O(n)\T)}{{\s}[O(n)]}{(\alpha_n)^*},
\]
is a Quillen equivalence. To show that the adjunction between the $S$-local model structures is a Quillen equivalence, it suffices by \cite[Proposition 2.3]{HoveySpectra} to show that if $Y$ is fibrant in ${\s}[O(n)]$ such that $(\alpha_n)^*Y$ is fibrant in the $S$-local $n$-stable model structure, then $Y$ is fibrant in the $S$-local model structure on ${\s}[O(n)]$. This follows readily from the definitions of fibrant objects in both model structures.  
\end{proof}

The category of orthogonal bispectra with an $O(n)$-action, or equivalently the category of (orthogonal) spectrum objects in spectra with an $O(n)$-action may be equipped with an $L_S$-local model structure, similar to Lemma \ref{prop: E-local intermediate}. For $S$ a set of maps of spectra, the $S$-local model structure $L_{S}\s(\s[O(n)])$ is the left Bousfield localization of the stable model structure at the set
\[
\{ \es{J}_1(V,-) \wedge j \mid V \in \es{J}, j \in J_{L_S\s[O(n)]}\}, 
\]
since the category $\s(\s[O(n)])$ may also be described as the category of $O(n)$-objects in $\Fun(\es{J}_1, \s)$. In particular, the fibrant objects of the $S$-local model structure on $\s(\s[O(n)])$ are $O(n)$-objects $X \in \Fun(\es{J}_1, \s)$ such that $X(V)$ is an $S$-local spectrum for each $V \in \es{J}_1$. 

\begin{thm}\label{thm: E-local spectra and spectral intermediate}
Let $S$ be a set of maps of spectra. The adjoint pairs
\[
\xymatrix@C+1ex{
L_S\Fun(\es{J}_n, \s[O(n)]) \ar@<1ex>[r]^{(\alpha_n)_!} & L_S\s(\s[O(n)]) \ar@<-1ex>[r]_{\rm{Ev}_0} \ar@<1ex>[l]^{(\alpha_n)^\ast} &  L_S\s[O(n)]  \ar@<-1ex>[l]_{F_0}
}
\]
are Quillen equivalences between the $S$-local model structures. 
\end{thm}
\begin{proof}
Identifying the category of spectrum objects in spectra with an $O(n)$-action with the category of $O(n)$-objects in $\Fun(\es{J}_1, \s)$, the proof that the adjunction
\[
\adjunction{(\alpha_n)_!}{L_S\Fun(\es{J}_n, \s[O(n)])}{L_S\s(\s[O(n)])}{(\alpha_n)^*},
\]	
is a Quillen equivalence follows analogously to Theorem \ref{thm: E-local spectra and intermediate}.

For the adjunction 
\[
\adjunction{F_0}{L_S\s[O(n)]}{L_S\s(\s[O(n)])}{\Ev_0},
\]
note that the composite functor 
\[
\s[O(n)] \xrightarrow{\ F_0 \ } \s(\s[O(n)]) \xrightarrow{\ \mathds{1} \ } L_S\s(\s[O(n)]),
\]
is left Quillen, and to extend to a left Quillen functor from $L_S\s[O(n)]$, it suffices by \cite[Proposition 3.3.18(1) $\&$ Theorem 3.1.6(1)]{Hi03}, to exhibit that the right adjoint preserves $S$-local objects, which follows immediately from the definition of $S$-local objects in the respective model structures. 

To see that the adjunction is a Quillen equivalence, we apply \cite[Proposition 2.3]{HoveySpectra}, which reduces the problem to showing that if $Y$ is an $\Omega$-spectrum object in $\s[O(n)]$ (i.e., fibrant in $\s(\s[O(n)])$) such that $\Ev_0(Y)$ is $S$-local, then $Y$ is $S$-local. This follows from the $\Omega$-spectrum structure and the interaction of homotopy function complexes with the suspension-loops adjunction.
\end{proof}

	\section{Homogeneous functors and their classification}\label{section: classification of $n$-homogeneous functors}

		\subsection{Homogeneous functors}

The layers of the Weiss tower associated to an orthogonal functor $F$ are the homotopy fibres of maps $T_{n}F \to T_{n-1}F$ and have two interesting properties: firstly, they are polynomial of degree less than or equal to $n$ and secondly, their $(n-1)$-polynomial approximation is trivial. We denote the $n$-th layer of the Weiss tower of $F$ by $D_nF$.

\begin{definition}
For $n \geq 0$, an orthogonal functor $F$ is said to be \emph{$n$-reduced} if its $(n-1)$-polynomial approximation is objectwise weakly equivalent to the terminal object. An orthogonal functor $F$ is said to be \emph{homogeneous of degree $n$} if it is both polynomial of degree less than or equal $n$ and $n$-reduced. We will sometimes refer to a functor which is homogeneous of degree $n$ as being \emph{$n$-homogeneous}.
\end{definition}

There is a model structure on the category of orthogonal functors which contains the $n$-homogeneous functors as the bifibrant objects. This model structure is a right Bousfield localization of the $n$-polynomial model structure.

\begin{prop}[{\cite[Proposition 6.9]{BO13}}]
There is a  model category structure on the category of orthogonal functors with weak equivalences the $D_n$-equivalences and fibrations the fibrations of the $n$-polynomial model structure. The cofibrant objects are the $n$-reduced projectively cofibrant objects and the fibrant objects are the $n$-polynomial functors. In particular, cofibrant-fibrant objects of this model structure are the projectively cofibrant $n$-homogeneous functors. This model structure is cellular, proper, stable and topological, and in the case of $\s$-valued orthogonal functors this model structure is spectral. We call this the $n$-homogeneous model structure and denote it by $\homog{n}(\es{J}_0, \es{C})$. 
\end{prop}

\begin{rem}
The model structure of \cite[Proposition 6.9]{BO13} has as weak equivalences those maps which induce objectwise weak equivalences on the $n$-th derivatives of their $n$-polynomial approximations. We showed in \cite[Proposition 8.2]{TaggartUnitary}, that the class of such equivalences is precisely the class of $D_n$-equivalences. The proof of \cite[Proposition 8.2]{TaggartUnitary} is valid for $\s$-valued orthogonal functors since $\s$-valued $n$-homogeneous functors admit and analogous classification in terms of spectral with an $O(n)$-action, see e.g., \cite[\S11]{BO13}. 
\end{rem}
	
The $n$-homogeneous model structure is (zigzag) Quillen equivalent to spectra with an action of $O(n)$. 

\begin{prop}[{\cite[Proposition 8.3, Theorem 10.1, Theorem 11.3 $\&$ Corollary 11.4]{BO13}}]
Let $n \geq 0$. There is a zigzag of Quillen equivalences
\[
\homog{n}(\es{J}_0, \es{C}) \simeq_Q {\s}[O(n)].
\]	
\end{prop}

On the homotopy category level the Barnes-Oman zigzag of Quillen equivalences recovers Weiss' characterisation of homogeneous functors of degree $n$.

\begin{prop}[{\cite[Theorem 7.3]{We95}},{\cite[Theorem 11.5]{BO13}}]\label{prop: classification of n-homog}
Let $n \geq 1$. 
\begin{enumerate}
	\item An $n$-homogeneous functor $F$ is determined by and determines a spectrum $\partial_nF$ with an $O(n)$-action.
	\item A $\T$-valued $n$-homogeneous functor $F$ is objectwise weak homotopy equivalent to the functor 
\[
V \longmapsto \Omega^\infty[ (S^{\bb{R}^n \otimes V} \wedge \partial_nF)_{hO(n)}],
\]
and any functor of the above form is homogeneous of degree $n$.
	\item A $\s$-valued $n$-homogeneous functor $F$ is objectwise weak homotopy equivalent to the functor 
\[
V \longmapsto (S^{\bb{R}^n \otimes V} \wedge \partial_nF)_{hO(n)},
\]
and any functor of the above form is homogeneous of degree $n$.
\end{enumerate}
\end{prop}

		\subsection{Local homogeneous functors}

\begin{definition}
Let $S$ be a set of maps in $\es{C}$. An orthogonal functor $F$ is \emph{$S$-locally homogeneous of degree $n$} if it is objectwise $S$-local and $n$-homogeneous.
\end{definition}

\begin{lem}\label{lem: E-local layers are E-locally n-homog}
Let $S$ be a set of maps of in $\es{C}$, and $F$ and orthogonal functor. For $n \geq 1$, there is a homotopy fibre sequence 
\[
D_n^SF \longrightarrow T_n^SF \longrightarrow T_{n-1}^SF,
\]
in which $D_n^S(F)$ is
\begin{enumerate}
	\item homogeneous of degree $n$; and,
	\item $S$-locally $n$-homogeneous if, in addition, the class of $S$-local objects is closed under sequential homotopy colimits. 
	\end{enumerate}
\end{lem}
\begin{proof}
By \cite[Lemma 5.5]{We95} the homotopy fibre of a map between $n$-polynomial functors is $n$-polynomial, hence $D_n^SF$ is $n$-polynomial. Applying $T_{n-1}$ to the homotopy fibre sequence, yields that the $(n-1)$-polynomial approximation of $D_n^SF$ is objectwise weakly contractible, proving $(1)$. 

For $(2)$, observe that the homotopy fibre of a map between $S$-local objects is $S$-local and when the class of $S$-local objects is closed under sequential homotopy colimits, $T_nL_SF$ is $S$-local for all $n$. 
\end{proof}

\begin{exs}\hspace{10cm}
	\begin{enumerate}
		\item For homological localization at the Eilenberg-Maclane spectrum associated to a subring $R$ of the rationals, $D_n^{HR}F$ is $HR$-locally $n$-homogeneous for any orthogonal functor $F$.
		\item For nullification at a based finite cell complex $W$, $D_n^WF$ is $W$-locally $n$-homogeneous for any $\T$-valued orthogonal functor $F$.
		\item For a spectrum $E$ whose associated localization of spectra is smashing, $D_n^EF$ is $E$-locally $n$-homogeneous for any $\s$-valued orthogonal functor. 
	\end{enumerate}
\end{exs}

\begin{prop}\label{prop: E local n-homog model}
Let $S$ be a set of maps in $\es{C}$. There is model category structure on the category of orthogonal functors with cofibrations the cofibrations of the $n$-homogeneous model structure and fibrant objects the $n$-polynomial functors whose $n$-th derivative is objectwise $S$-local in the $n$-th intermediate category. This model structure is cellular, left proper and topological, and in the case of $\s$-valued orthogonal functors this model structure is spectral. We call this the $S$-local $n$-homogeneous model structure and denote it by $\homog{n}(\es{J}_0, L_S\es{C})$. 
\end{prop}
\begin{proof}
We left Bousfield localize the $n$-homogeneous model structure at the set of maps 
\[
\es{K}_n = \{ \es{J}_n(U,-) \wedge j \mid U \in \es{J}, j \in J_{L_S\es{C}} \}.
\]
This left Bousfield localization exists since the $n$-homogeneous model structure is cellular and left proper by \cite[Lemma 6.1]{Ba17}. The description of the cofibrations follows immediately. 

The fibrant objects are the $\es{K}_n$-local objects which are also fibrant in the $n$-homogeneous model structure, i.e., those $n$-polynomial functors $Z$ for which the induced map
\[
[\es{J}_n(U,-) \wedge B , Z] \longrightarrow [\es{J}_n(U,-) \wedge A, Z],
\]
is an isomorphism for all maps $\es{J}_n(U,-) \wedge A \to \es{J}_n(U,-) \wedge B$ in $\es{K}_n$. A straightforward adjunction argument and the definition of the $n$-th derivative of an orthogonal functor yield the required characterisation of the fibrant objects. 
\end{proof}

\begin{cor}
Let $S$ be a set of maps in $\es{C}$. The cofibrant objects of the $S$-local $n$-homogeneous model structure are the projectively cofibrant functors which are $n$-reduced.	
\end{cor}
\begin{proof}
	The $S$-local $n$-homogeneous model structure is a particular left Bousfield localization of the $n$-homogeneous model structure, hence has the same cofibrant objects. The result follows by the orthogonal calculus version of \cite[Corollary 8.6]{TaggartUnitary}.
\end{proof}

The $S$-local $n$-homogeneous model structure behaves like a right Bousfield localization of the $S$-local $n$-polynomial model structure in the following sense.

\begin{lem}
Let $S$ be a set of maps in $\es{C}$. The adjoint pair
\[
\adjunction{ \mathds{1}}{\homog{n}(\es{J}_0,L_S\es{C})}{\poly{n}(\es{J}_0, L_S\es{C})}{ \mathds{1}},
\]
is a Quillen adjunction. 
\end{lem}
\begin{proof}
The cofibrations of the $S$-local $n$-homogeneous model structure are the cofibrations of the $n$-homogeneous model structure, which are contained in the cofibrations of the $n$-polynomial model structure, which in turn are precisely the cofibrations of the $S$-local $n$-polynomial model structure, hence 
\[
\mathds{1} \colon \homog{n}(\es{J}_0, L_S\es{C}) \longrightarrow \poly{n}(\es{J}_0, L_S\es{C}),
\]
preserves cofibrations.

On the other hand, to show that the right adjoint is right Quillen it suffices to show that the identity functor sends fibrant objects in the $S$-local $n$-polynomial model structure to fibrant objects in the $S$-local $n$-homogeneous model structure. This follows from Lemma \ref{lem: derivative E-local} since the fibrant objects in the $S$-local $n$-polynomial model structure are the $S$-locally $n$-polynomial functors by Proposition \ref{prop: S-local n-poly model structure} and the fibrant objects of the $S$-local $n$-homogeneous model structure are the $n$-polynomial functors with $S$-local $n$-th derivative by Proposition \ref{prop: E local n-homog model}.
\end{proof}

		\subsection{Characterisations for stable localizations}

We obtain a characterisation of the fibrations of the $S$-local $n$-homogeneous model structure when the localizing set $S$ is stable in the sense of \cite[Definition 4.2]{BarnesRoitzheimStable}, i.e., when the class of $S$-local spaces is closed under suspension.  For the statement of the following result recall the definition of the $n$-th derivative of an orthogonal functor from Definition~\ref{def: derivative}.

\begin{prop}\label{prop: fibs of E-local n-poly}
If $S$ is a set of maps in $\es{C}$ which is stable, then the fibrations of the $S$-local $n$-homogeneous model structure are those maps $f \colon X \to Y$ which are fibrations in the $n$-polynomial model structure such that 
\[
X^{(n)} \longrightarrow Y^{(n)},
\]
is a objectwise fibration in $L_S\es{C}$.
\end{prop}
\begin{proof}
We first given an explicit characterisation of the acyclic cofibrations since the fibrations are characterised by the right lifting property against these maps. The maps in $\es{K}_n$ are cofibrations between cofibrant objects since $\es{J}_n(U,-)$ is cofibrant in $\homog{n}(\es{J}_0, \es{C})$ and the maps in $J_{L_S\es{C}}$ are cofibrations of the $S$-local model structure on $\es{C}$. Moreover, since the localizing set $S$ is stable, it follows the set of generating acyclic cofibrations $J_{L_S\es{C}}$ is stable and in turn that the set $\es{K}_n$ is stable. Hence by \cite[Theorem 4.11]{BarnesRoitzheimStable}, the generating acyclic cofibrations are given by the set $J_{\homog{n}} \cup \Lambda(\es{K}_n)$, where $J_{\homog{n}}$ is the set of the generating acyclic cofibrations of the $n$-homogeneous model structure and $\Lambda(\es{K}_n)$ the set of horns on $\es{K}_n$ in the sense of \cite[Definition 4.2.1]{Hi03}. As horns in topological model categories are given by pushouts and $\es{K}_n$ is a set of cofibrations between cofibrant objects it suffices to use the set $J_{\homog{n}} \cup \es{K}_n$ as the generating acyclic cofibrations of the $S$-local $n$-homogeneous model structure.

If $f \colon X \to Y$ is a map with the right lifting property with respect to $J_{\homog{n}} \cup \es{K}_n$, then $f$ has the right lifting property with respect to $J_{\homog{n}}$ and the right lifting property with respect to $\es{K}_n$ independently. Having the right lifting property with respect to $J_{\homog{n}}$ is equivalent to being a fibration in the $n$-polynomial model structure. On the other hand, a map in $\es{K}_n$ is of the form $\es{J}_n(U,-) \wedge A \to \es{J}_n(U,-) \wedge B$ for $A \to B$ a generating acyclic cofibration of the $S$-local model structure on $\es{C}$. A lift in the diagram 
\[
\xymatrix{
\es{J}_n(U,-) \wedge A \ar[r] \ar[d] & X  \ar[d] \\
\es{J}_n(U,-) \wedge B \ar[r] \ar@{-->}[ur] & Y \\
}
\]
(indicated by the dotted arrow) exists if and only if the lift in the diagram
\[
\xymatrix{
A \ar[r] \ar[d] & \nat_0(\es{J}_n(U,-) , X)  \ar[d] \\
B \ar[r] \ar@{-->}[ur] & \nat_0(\es{J}_n(U,-) ,Y) \\
}
\]
exists, which is equivalent to the statement that $X^{(n)} \to Y^{(n)}$ is a objectwise fibration of $S$-local objects in $\es{C}$, see \S\ref{ssection: derivatives}.
\end{proof}

This specialises to homological localizations.

\begin{cor}\label{cor: Fibrations in E-local n-homog}
Let $E$ be a spectrum. The fibrations of the $E$-local $n$-homogeneous model structure are those maps $f \colon X \to Y$ which are fibrations in the $n$-polynomial model structure  such that 
\[
X^{(n)} \longrightarrow Y^{(n)},
\]
is a objectwise fibration in $L_E\es{C}$.
\end{cor}
\begin{proof}
Combine Proposition \ref{prop: fibs of E-local n-poly} with \cite[Example 4.3]{BarnesRoitzheimStable}.	
\end{proof}

\begin{cor}\label{cor: (co)fibrants in E-local n-homog}
Let $E$ be a spectrum. An orthogonal functor $F$ is fibrant in the $E$-local $n$-homogeneous model structure if and only if $F$ is $n$-polynomial and $F^{(n)}$ is objectwise $E$-local. In particular, the bifibrant objects are the projectively cofibrant $n$-homogeneous functors with $E$-local $n$-th derivative. 
\end{cor}
\begin{proof}
Apply Corollary \ref{cor: Fibrations in E-local n-homog} to the map $F \to \ast$.  
\end{proof}

		\subsection{Differentiation as a Quillen functor}

The $n$-th derivative is a right Quillen functor as part of a Quillen equivalence between the $n$-homogeneous model structure and the $n$-th intermediate category; the adjunction
\[
\adjunction{\res_0^n/O(n)}{{\Fun}_{O(n)}(\es{J}_n, \es{C}[O(n)])}{\homog{n}(\es{J}_0, \es{C})}{{\ind}_0^n\varepsilon^*},
\]
is a Quillen equivalence, \cite[Theorem 10.1]{BO13}. We now show that this extends to the $S$-local situation.

\begin{thm}\label{thm: E-local diff as Quillen}
Let $S$ be a set of maps in $\es{C}$. The adjoint pair
\[
\adjunction{\res_0^n/O(n)}{L_S{\Fun}_{O(n)}(\es{J}_n, \es{C}[O(n)])}{\homog{n}(\es{J}_0, L_S\es{C})}{{\ind}_0^n\varepsilon^*},
\]
is a Quillen equivalence between the $S$-local model structures. 
\end{thm}
\begin{proof}
The left adjoint applied to the localizing set of the $S$-local $n$-stable model structure is precisely the localization set of the $S$-local $n$-homogeneous model structure, hence the result follows from \cite[Theorem 3.3.20(1)]{Hi03}.
\end{proof}

\begin{cor}\label{cor: zigzag}
Let $S$ be a set of maps of based spaces, and $n \geq 0$. There is a zigzag of Quillen equivalences
\[
\homog{n}(\es{J}_0, L_S\T) \simeq_Q {\s}(L_S\T)[O(n)].
\]
\end{cor}

\begin{ex}
Let $R$ be a subring of the rationals. Then there is a zigzag of Quillen equivalences	
\[
\homog{n}(\es{J}_0, L_{HR}\T) \simeq_Q {\s}_{HR}[O(n)],
\]
between $HR$-local $n$-homogeneous functors and $HR$-local\footnote{In particular, the $HR$-local model structure on spectra is identical to the levelwise $HR$-local model structure since a spectrum is $HR$-local if and only if it is levelwise $HR$-local, see e.g., \cite[Lemma 8.6]{BarnesRoitzheimFramings}.} spectra with an action of $O(n)$. 
\end{ex}

\begin{cor}\label{cor: zigzag spectra}
Let $S$ be a set of maps of spectra, and $n \geq 0$. There is a zigzag of Quillen equivalences
\[
\homog{n}(\es{J}_0, L_S\s) \simeq_Q L_S\s[O(n)].
\]
\end{cor}

\begin{ex}
Let $E$ be a spectrum. Then there is a zigzag of Quillen equivalences	
\[
\homog{n}(\es{J}_0, L_{E}\s) \simeq_Q {\s}_{E}[O(n)],
\]
between $E$-local $n$-homogeneous functors and $E$-local spectra with an action of $O(n)$. 
\end{ex}

		\subsection{The classification}

As in the classical theory, any $S$-locally $n$-homogeneous functor may be expressed concretely in terms of a levelwise $S$-local spectrum with an action of $O(n)$. The proof of which follows as in the classical setting, \cite[Theorem 7.3]{We95} and can be realised through the derived equivalence of homotopy categories provided by our zigzag of Quillen equivalences.

\begin{thm}\label{thm: classification of E-local n-homog}\label{thm: classification of spectral E-local n-homog}
Let $S$ be a set of maps of in $\es{C}$ and $n \geq 1$. 
\begin{enumerate}
\item An $S$-local $n$-homogeneous functor $F$ is determined by and determines an appropriately $S$-local spectrum with an $O(n)$-action, denoted $\partial_n^S F$.
\item A $\T$-valued $S$-local $n$-homogeneous functor $F$ is objectwise weakly equivalent to the functor 
\[
V \longmapsto \Omega^\infty[ (S^{\bb{R}^n \otimes V} \wedge \partial_n^S F)_{hO(n)}],
\]
and any functor of the above form is objectwise $S$-local and $n$-homogeneous. 
\item A $\s$-valued $S$-local $n$-homogeneous functor $F$ is objectwise weakly equivalent to the functor
\[
V \longmapsto (S^{\bb{R}^n \otimes V} \wedge \partial_n^S F)_{hO(n)},
\]
and any functor of the above form is objectwise $S$-local and $n$-homogeneous. 
\end{enumerate}
\end{thm}

\part{Applications}

	\section{Bousfield classes}

		\subsection{Bousfield classes}

For a spectrum $E$, the \emph{Bousfield class} of $E$, denoted $\langle E \rangle$, is the equivalence class of $E$ under the relation: $E \sim E'$ if for any spectrum $X$, $E \wedge X =0$ if and only if $E'\wedge X =0$. If $\langle E \rangle = \langle E' \rangle$, then the classes of $E_\ast$-isomorphisms and $E_\ast'$-isomorphisms agree and hence the localization functors (on spaces or spectra) agree. The collection of all Bousfield classes forms a lattice, with partial ordering $\langle E \rangle \leq \langle E' \rangle$ given by reverse containment, i.e., if and only if the class of $E'$-acyclic spectra is contained in the class of $E$-acyclic spectra, in particular, the partial ordering induces a natural transformation $L_{E'} \to L_E$. Bousfield classes have been studied at length, see e.g, \cite{BousfieldLocalSpectra, RavenelLocalizations}.

A similar story remains true unstably. Given a based space $W$ the \emph{unstable Bousfield class} of $W$, or the \emph{nullity class} of $W$, is the equivalence class $\langle W \rangle$ of all spaces $W'$ such that the class of $W$-periodic\footnote{$W$-periodic spaces are precisely $W$-local spaces. This change in terminology is classical, see e.g, \cite{BousfieldPeriodicity, FarjounCellular}.} spaces agrees with the class of $W'$-periodic spaces. There is a partial ordering $\langle W \rangle \leq \langle W' \rangle$ given by reverse containment, i.e., if and only if every $W'$-periodic space is $W$-periodic. In particular, the relation $\langle W \rangle \leq \langle W' \rangle$ implies that every $W$-local equivalence is a $W'$-local equivalence and there is a natural transformation $P_W \to P_{W'}$, which is a $W'$-localization. Nullity classes have also been studied at length, see e.g, \cite{BousfieldPeriodicity, FarjounCellular}. 

\begin{rem}
It is worth noting that in both cases there is a choice of ordering of the equivalence classes, and our choices have been made to align with the predominant references on the subject, which unfortunately means the ''stable'' and ``unstable'' directions are dual. The choice of ordering used by Bousfield and that of Farjoun also differ, adding further confusion to the literature on these matters. 
\end{rem}

\begin{thm}\label{thm: BF classes iff towers}
Let $S$ and $S'$ be sets of maps in $\es{C}$. The class of $S$-local objects agrees with the class of $S'$ local objects if and only if for every orthogonal functor $F$, the $S$-local Weiss tower of $F$ is objectwise weakly equivalent to the $S'$-local Weiss tower of $F$.
\end{thm}
\begin{proof}
If the class of $S$-local objects agrees with the class of $S'$-local objects, then the localization functors $L_S$ and $L_{S'}$ agree on $\es{C}$ and hence on the level of orthogonal functors. In particular, for every orthogonal functor $F$, the canonical map\footnote{This map is induced from the $S$-local objects being contained in the $S'$-local objects. We could also use the canonical $L_{S'}F \to L_{S}F$ since the $S$-local objects also contained the ${S'}$-local objects.} 
\[
L_S F \to L_{S'}F,
\] 
is a objectwise weak equivalence. Now, consider the commutative diagram
\[
\xymatrix{
D_n^S F \ar[r] \ar[d] & T_n^S F \ar[r] \ar[d] & T_{n-1}^S F \ar[d] \\
D_n^{S'} F \ar[r] & T_n^{S'} F \ar[r] & T_{n-1}^{S'}F \\
}
\]
in which the rows are homotopy fibre sequences. For each $n \geq 0$, the map
\[
T_n^S F \longrightarrow T_n^{S'}F,
\]
is a objectwise weak equivalence since polynomial approximation preserves objectwise weak equivalences. It follows that the left-most vertical arrow is also a objectwise weak equivalence and that the $S$-local Weiss tower is objectwise weakly equivalent to the $S'$-local Weiss tower.

The converse is immediate from specialising for every object $C \in \es{C}$ to the constant functor at $C$. 
\end{proof}

\begin{exs} \hspace{10cm}
\begin{enumerate}
	\item Let $E$ and $E'$ be spectra. For every orthogonal functor $F$ the $E$-local Weiss tower of $F$ and the $E'$-local Weiss tower of $F$ agree if and only if $E$ and $E'$ are Bousfield equivalent.
	\item Let $W$ and $W'$ be based spaces. For every $\T$-valued orthogonal functor $F$ the $W$-local Weiss tower of $F$ and the $W'$-local Weiss tower of $F$ agree if and only if $W$ and $W'$ have the same nullity class.	
\end{enumerate}
\end{exs}

		\subsection{Bousfield classes and model categories for orthogonal calculus}

On the model category level, we have the following. 

\begin{thm}\label{thm: Bousfield classes homog QE}
Let $S$ and $S'$ be sets of maps of maps in $\es{C}$. The class of $S$-local objets in $\es{C}$ agrees with the class of $S'$-local objects in $\es{C}$ if and only if there are equalities of model structures making the diagram
\[
\xymatrix{
{\Fun}(\es{J}_0, L_S\es{C}) \ar@{=}[d] \ar@<1ex>[r]^{\mathds{1}} & \poly{n}(\es{J}_0, L_S\es{C}) \ar@<-1ex>[r]_{\mathds{1}} \ar@<1ex>[l]^{\mathds{1}} \ar@{=}[d] & \homog{n}(\es{J}_0, L_S\es{C}) \ar@<-1ex>[l]_{\mathds{1}} \ar@{=}[d] \\
{\Fun}(\es{J}_0, L_{S'}\es{C}) \ar@<1ex>[r]^{\mathds{1}} & \poly{n}(\es{J}_0, L_{S'}\es{C})  \ar@<-1ex>[r]_{\mathds{1}} \ar@<1ex>[l]^{\mathds{1}} & \homog{n}(\es{J}_0, L_{S'}\es{C}) \ar@<-1ex>[l]_{\mathds{1}}
}
\]
commute.
\end{thm} 
\begin{proof}
For one direction assume that the class of $S$-local objects agrees with the class of $S'$-local objects. Then the $S$-local model structure and the $S'$-local model structure on $\es{C}$ agree as they have the same cofibrations and fibrant objects. This equality lifts to the local projective model structures on the category of orthogonal functors. As left Bousfield localization does not alter the cofibrations, the cofibrations of the $S$-local $n$-polynomial model structure agree with the cofibrations of the $S'$-local $n$-polynomial model structure. These model structures also have the same fibrant objects since a functor is $S$-locally $n$-polynomial if and only if it is $S'$-local $n$-polynomial under our assumption. 

For the local $n$-homogeneous model structures, recall that these are certain left Bousfield localizations of the $n$-homogeneous model structure (see Proposition \ref{prop: E local n-homog model}), hence have the same cofibrations. As before, these model structures have the same fibrant objects since our assumption together with Lemma \ref{lem: derivative E-local} implies that the $n$-th derivative of a functor is $S$-local if and only if it is $S'$-local, and the fibrant objects are the $n$-polynomial functors with local derivatives, see Proposition \ref{prop: E local n-homog model}.

For the converse note that since the $S$-local model structure on the category of orthogonal functors agrees with the $S'$-local model structure, the objectwise $S$-local equivalences are precise the objectwise $S$-local equivalences. It follows that the local model structures on $\es{C}$ must agree. 
\end{proof}

		\subsection{The partial ordering of Bousfield classes}

\begin{lem}\label{lem: inequality of unstable Bousfield classes}
Let $S$ and $S'$ be sets of maps in $\es{C}$ and $F$ an orthogonal functor. If the class of $S'$-local objects of $\es{C}$ is contained in the class of $S$-local objects then,
\begin{enumerate}
	\item there is an $S'$-local equivalence $D_n^{S} F \to D_n^{S'}F$; and,
	\item if $F$ is reduced, then the $S$-local Weiss tower of $F$ is $S'$-locally equivalent to the $S'$-local Weiss tower of $F$.
\end{enumerate}
\end{lem}
\begin{proof}
For (1), note that the map on derivatives $\partial_n^{S} F \to \partial_n^{S'} F$ induced by the natural transformation $L_{S} \to L_{S'}$ is an $S'$-local equivalence, hence the $n$-homogeneous functors which correspond to these spectra are $S'$-locally equivalent, i.e., the map $D_n^{S} F \to D_n^{S'}F$ is an $S'$-local equivalence. For (2), since $F$ is reduced \cite[Corollary 8.3]{We95} implies that there is a commutative diagram
\[
\xymatrix{
T_n^{S} F \ar[r] \ar[d] & T_{n-1}^{S} F  \ar[r] \ar[d] & \ar[d] R_n^{S} F \\
T_n^{S'}F \ar[r] & T_{n-1}^{S'}F \ar[r] & R_n^{S'} F
}
\]
in which both rows are homotopy fibre sequences. The map $R_n^{S} F \to R_n^{S'}F$ is an $S'$-local equivalence by part (1), and the map $T_0^{S} F \to T_0^{S'} F$ is also an $S'$-local equivalence since $F$ is reduced. An induction argument on the degree of polynomials yields the result. 
\end{proof}

\begin{exs}\label{ex: BF classes}\hspace{10cm}
\begin{enumerate}
	\item Let $E$ and $E'$ be spectra and $F$ an orthogonal functor. If $\langle E \rangle \leq \langle E'\rangle$, then
\begin{enumerate}
	\item there is an $E$-local equivalence $D_n^{E'} F \to D_n^{E}F$; and,
	\item if $F$ is reduced, then the $E'$-local Weiss tower of $F$ is $E$-locally equivalent to the $E$-local Weiss tower of $F$.
\end{enumerate}
	\item Let $W$ and $W'$ be based spaces and $F$ a $\T$-valued orthogonal functor. If $\langle W \rangle \leq \langle W'\rangle$, then
\begin{enumerate}
	\item there is an $W'$-local equivalence $D_n^{W} F \to D_n^{W'}F$; and,
	\item if $F$ is reduced, then the $W$-local Weiss tower of $F$ is $W'$-locally equivalent to the $W'$-local Weiss tower of $F$.
\end{enumerate}
\end{enumerate}	
\end{exs}

		\subsection{The Telescope Conjecture}

The height $n$ Telescope Conjecture of Ravenel \cite[Conjecture 10.5]{RavenelLocalizations} asserts that the $T(n)$-localization and $K(n)$-localization of spectra agree. There are numerous equivalent formalisations of the conjecture see e.g., \cite[Proposition 3.6]{BarthelChromaticConjectures} and we choose the following as it best suits any possible interaction with the calculus. 

\begin{conj}[The height $n$ Telescope Conjecture]
	Let $n \geq 0$. The Bousfield class of $T(n)$ agrees with the Bousfield class of $K(n)$.
\end{conj}

\begin{cor}
Let $n \geq 0$. The validity of the height $n$ Telescope Conjecture implies equality of model structures
\[
\xymatrix{
{\Fun}(\es{J}_0, L_{T(n)}\es{C}) \ar@{=}[d] \ar@<1ex>[r]^{\mathds{1}} & \poly{n}(\es{J}_0, L_{T(n)}\es{C}) \ar@<-1ex>[r]_{\mathds{1}} \ar@<1ex>[l]^{\mathds{1}} \ar@{=}[d] & \homog{n}(\es{J}_0, L_{T(n)}\es{C}) \ar@<-1ex>[l]_{\mathds{1}} \ar@{=}[d] \\
{\Fun}(\es{J}_0, L_{K(n)}\es{C}) \ar@<1ex>[r]^{\mathds{1}} & \poly{n}(\es{J}_0, L_{K(n)}\es{C})  \ar@<-1ex>[r]_{\mathds{1}} \ar@<1ex>[l]^{\mathds{1}} & \homog{n}(\es{J}_0, L_{K(n)}\es{C}) \ar@<-1ex>[l]_{\mathds{1}}
}
\]
\end{cor}
\begin{proof}
	The Telescope Conjecture implies that the Bousfield class of $T(n)$ and the Bousfield class of $K(n)$, agree, hence the result follows by Theorem \ref{thm: Bousfield classes homog QE}.
\end{proof}

The following is an immediate corollary to Theorem \ref{thm: BF classes iff towers}.

\begin{cor}\label{cor: telescope implies agree towers}
Let $n \geq 0$. The height $n$ Telescope Conjecture holds if and only if for every orthogonal functor $F$ the $K(n)$-local Weiss tower of $F$ and the $T(n)$-local Weiss tower of $F$ agree.
\end{cor}

This provides new insight into the the height $n$ Telescope Conjecture. For example, to find a counterexample it now suffices to find an orthogonal functor such that one corresponding term in the $K(n)$-local and $T(n)$-local Weiss towers disagree. This can also be seen through the spectral sequences associated to the local Weiss towers. The $K(n)$-local and $T(n)$-local Weiss towers of an orthogonal functor $F$ produce two spectral sequences, 
\[
\pi_{t-s}D_s^{K(n)}F(V) \cong \pi_{t-s} ((S^{\bb{R}^s \otimes V} \wedge \partial_s^{K(n)}F)_{hO(n)}) \Rightarrow \pi_\ast \underset{d}{\holim}~ T_d^{K(n)}F(V),
\]
and, 
\[
\pi_{t-s}D_s^{T(n)}F(V) \cong \pi_{t-s} ((S^{\bb{R}^s \otimes V} \wedge \partial_s^{T(n)}F)_{hO(n)}) \Rightarrow \pi_\ast \underset{d}{\holim}~ T_d^{T(n)}F(V),
\]
These are closely related to the Telescope Conjecture as follows.

\begin{lem}\label{lem: telescope and WSS}
Let $F$ be an orthogonal functor. If the height $n$ Telescope Conjecture holds, then for all $r \geq 1$, the $E_r$-page of the $T(n)$-local Weiss spectral sequence is isomorphic to the $E_r$-page of the $K(n)$-local Weiss spectral sequence.
\end{lem}
\begin{proof}
It suffices to prove the claim for $r=1$. The validity of the height $n$ Telescope Conjecture implies that there is a natural transformation $L_{K(n)} \to L_{T(n)}$. This natural transformation induces a map $D_d^{K(n)}F \to D_d^{T(n)}F$, which by Corollary \ref{cor: telescope implies agree towers} is an objectwise weak equivalence. It hence suffices to show that the natural map $D_d^{K(n)}F \to D_d^{T(n)}F$ induces a map on the $E_1$-pages of the spectral sequences, that is, we have to show that the induced diagram
\[
\xymatrix{
\pi_{t-s}D_s^{K(n)}F(V) \ar[r]^{d_1^{K(n)}} \ar[d] & \pi_{t-s+1}D_{s+1}^{K(n)}F(V) \ar[d] \\
\pi_{t-s}D_s^{T(n)}F(V) \ar[r]_{d_1^{T(n)}} & \pi_{t-s+1}D_{s+1}^{T(n)}F(V) \\
}
\]
commutes for all $s$ and $t$. This follows from the commutativity of the induced diagram of long exact sequences induced by the diagram of homotopy fibre sequences, 
\[
\xymatrix{
D_s^{K(n)}F(V) \ar[r] \ar[d] & T_s^{K(n)}F(V) \ar[r] \ar[d] & T_{s-1}^{K(n)}F(V) \ar[d] \\
D_s^{T(n)}F(V) \ar[r] & T_s^{T(n)}F(V) \ar[r] & T_{s-1}^{T(n)}F(V) 
}
\]
and the construction of the $d_1$-differential in the homotopy spectral sequence associated to a tower of fibrations. 
\end{proof}

	\section{The calculus for nullifications}

		\subsection{Nullifications of orthogonal functors}
Bousfield, Farjoun and others, see e.g., \cite{BousfieldPeriodicity, BousfieldUnstable, FarjounCellular, CasacubertaUnstable}, have extensively studied the nullification of the category of based spaces at a based space $W$. This nullification is functorial giving a functor
\[
P_W \colon \T \longrightarrow \T,
\]
and the Bousfield-Friedlander localization of $\T$ at the endofunctor $P_W$ defines a model structure which we call the $W$-periodic model structure, and denote by $P_W\T$. This model structure is precisely the left Bousfield localization at the set $S = \{\ast \to W\}$, i.e., the $W$-periodic and $W$-local model structures agree. 

The endofunctor $P_W : \T \to \T$ extends objectwise to a functor
\[
P_W \colon \Fun(\es{J}_0, \T) \longrightarrow \Fun(\es{J}_0, \T),
\]
and the $W$-periodic model structure on spaces (see, e.g., \cite[\S\S9.8]{Bo01}) extends in a canonical way to give the Bousfield-Friedlander localization of the category of orthogonal functors at the functor $P_W$, which we denote by ${\Fun}(\es{J}_0, P_W\T)$, and call the $W$-periodic model structure. This model structure agrees with the $S$-local model structure on orthogonal functors for $S=\{\ast \to W\}$.

In this section we give an alternative construction of the model structures for $W$-local orthogonal calculus. The key to this is that the $W$-periodic model structure on based spaces is right proper.

\begin{rem}\label{rem: example of not right proper}
The process of left Bousfield localization can interfere with other model categorical properties, for instance left Bousfield localization need not preserve right properness. For example if $E = H\bb{Q}$, then the $H\bb{Q}$-local model structure on  based spaces is not right proper since there is a pullback square
\[
\xymatrix{
K(\bb{Q}/\bb{Z},0) \ar[r] \ar[d] & P \ar[d] \\
K(\bb{Z}, 1) \ar[r]^{\simeq H\bb{Q}} & K(\bb{Q},1)
}
\]
in which the right hand vertical map is a fibration, $P$ is contractible and the lower horizontal map is a $H\bb{Q}$-equivalence but the left hand vertical map is not. Another example is provided by Quillen in \cite[Remark 2.9]{QuillenRational}. 
\end{rem}

The property of being right proper has many advantages including the ability to right Bousfield localize. As such we investigate when the $S$-local model structure is right proper. It suffices to examine when the $f$-local model structure is right proper for some map $f \colon X \to Y$ of based spaces.

The following has motivation in \cite[Remark 9.11]{Bo01}, in which Bousfield remarks that the $f$-local model structure cannot be right proper unless the localization functor $L_f$ is equivalent to a nullification. We extend Bousfield's remark by showing that his nullification condition is both necessary and sufficient in a stronger sense than originally proposed by Bousfield. This result depends on two constructions also due to Bousfield; the first is the construction of a based space $A(f)$ associated to a map $f \colon X \to Y$ of based spaces, see \cite[Theorem 4.4]{BousfieldHomotopical}, the second is the nullification functor $P_W \colon \T \to \T$ associated to any based space $W$, see \cite[Theorem 2.10]{BousfieldPeriodicity}. This nullification functor has two key properties which we would also like to highlight; firstly, when $W$ is connected $P_W$ preserves disjoint unions, e.g.,~\cite[Theorem 9.9]{Bo01}, and secondly, $P_W$ is contractible when $W$ is not connected, see e.g.,~\cite[Example 2.3]{BousfieldPeriodicity}. For example, if $f$ is the map which induces localization with respect to integral homology, then $P_{A(f)}$ is Quillen's plus construction, see e.g., \cite[1.E.5]{FarjounCellular}.

\begin{prop}\label{prop: local model as a nullification}
Let $f \colon X \to Y$ be a map of based spaces. The $f$-local model structure on based spaces is right proper if and only if there exists a based space $A(f)$ and and equality of model structures
\[
L_f\T = P_{A(f)} \T,
\]
where $P_{A(f)}\T$ is the Bousfield-Friedlander localization \cite[Theorem 9.3]{Bo01}, at the nullification endofunctor 
\[
P_{A(f)} \colon \T \to \T.
\]
\end{prop}
\begin{proof}
By \cite[Theorem 4.4]{BousfieldHomotopical}, there exists a based space $A(f)$ such that the classes of $A(f)$-acyclic  and $f$-acyclic spaces agree, and every $P_{A(f)}$-equivalence is an $f$-local equivalence.

Assume that the $f$-local model structure is right proper. For a connected based space $X$, the path fibration over $L_fX$ is an $f$-local fibration, hence the homotopy fibre of the map $X \to L_f X$ is $f$-acyclic, and hence $A(f)$-acyclic. It follows by \cite[Corollary 4.8(i)]{BousfieldPeriodicity}, the map $X \to L_f X$ is a $P_{A(f)}$-equivalence, hence every $f$-local equivalences of connected spaces is a  $P_{A(f)}$-equivalence. Since the functor $P_{A(f)}$ on based spaces comes from a functor on unbased spaces which preserves disjoint unions when $A(f)$ is connected and which takes contractible values when $A(f)$ is not connected, every $f$-local equivalence must be a $P_{A(f)}$-equivalence. It follows that the class of $f$-local equivalences agrees with the class of $P_{A(f)}$-equivalences. The equality of the model structures follows immediately since both model structures have the same cofibrations inherited from the Quillen model structure on the category of based spaces.

For the converse, assume that the $f$-local model structure agrees with the $A(f)$-local model structure. The latter model structure is right proper by \cite[Theorem 9.9]{Bo01}, and since both model structures have the same weak equivalences and fibrations, the $f$-local model structure must also be right proper. 
\end{proof}

\begin{rem}
The property of being right proper is completely determined by the weak equivalence class of the model structure; if two model structures have the the same weak equivalences, then one is right proper if and only if the other is, see e.g., \cite[Remark 2.5.6]{BalchinModelCats}. 
\end{rem}

		\subsection{Nullifications and polynomial functors}
Recall from Proposition \ref{prop: S-local n-poly model structure} that we have minimal control over the $W$-local $n$-polynomial model structure, in particular, unless the localization is well-behaved with respect to sequential homotopy colimits, $T_nL_W$ is not a fibrant replacement functor. We construct a $W$-periodic $n$-polynomial model structure as the Bousfield-Friedlander localization at the composite
\[
T_n \circ P_W \colon {\Fun}(\es{J}_0, \T) \longrightarrow {\Fun}(\es{J}_0, \T).
\]
and show that this model structure is precisely the $W$-local $n$-polynomial model structure.

We begin with a lemma which deals with fibrant objects in the Bousfield-Friedlander localization of orthogonal functors at the endofunctor $P_W$, which we call the $W$-periodic projective model structure. 

\begin{lem}\label{lem: fib in null model}
For a finite cell complex $W$ and an orthogonal functor $F$, the functor $T_nP_WF$ is fibrant in the Bousfield-Friedlander localization of the category of orthogonal functors at the functor $P_W$. In particular, the map
\[
\omega_{T_nP_WF} \colon T_nP_WF \longrightarrow P_WT_nP_WF,
\] 
is a objectwise weak homotopy equivalence. 
\end{lem}
\begin{proof}
The Bousfield-Friedlander localization of based spaces at the endofunctor $P_W$ is identical to the left Bousfield localization of based spaces at the map $\ast \to W$, since both model structures have the same cofibrations and fibrant objects. It follows that the Bousfield-Friedlander localization of the category of orthogonal functors at the endofunctor $P_W$ is identical to the $W$-local projective model structure. In particular, we see that $P_WF$ is fibrant and hence $\tau_n P_WF$ is also fibrant, since the class of $W$-local objects is closed under homotopy limits.  The result follows since local objects for a nullification are closed under sequential homotopy colimits by \cite[1.D.6]{FarjounCellular}.
\end{proof}

\begin{prop}\label{prop: BF localization at composite}
For a finite cell complex $W$ the Bousfield-Friedlander localization of the category of orthogonal functors at the endofunctor
\[
T_n \circ P_W \colon \Fun(\es{J}_0, \T) \longrightarrow \Fun(\es{J}_0, \T),
\]
exists. This model structure is proper and topological. We call this the $W$-periodic $n$-polynomial model structure and denote it by $\poly{n}(\es{J}_0, P_W\T)$. 
\end{prop}
\begin{proof}
We verify the axioms of \cite[Theorem 9.3]{Bo01}. First note that since $P_W$ and $T_n$ both preserve objectwise weak equivalences so does their composite, hence verifying \cite[Theorem 9.3(A1)]{Bo01}.  

The natural transformation from the identity to the composite $T_n \circ P_W$ is given in components as the composite
\[
F \xrightarrow{\omega_F} P_W F \xrightarrow{\eta_{P_WF}} T_nP_W F, 
\]
where $\omega \colon \mathds{1} \to P_W$ and $\eta \colon \mathds{1} \to T_n$, hence at $T_nP_WF$, we obtain the composite
\[
T_nP_WF \xrightarrow{\omega_{T_nP_WF}} P_WT_nP_W F \xrightarrow{\eta_{P_WT_nP_WF}} T_nP_W T_nP_WF. 
\]
Since the domain is fibrant in the $W$-periodic projective model structure the first map in the composite is a objectwise weak equivalence, see Lemma \ref{lem: fib in null model}. The second map is also a weak equivalence. To see this, note that since $T_nP_WF$ is polynomial of degree less than or equal $n$, the functor $P_WT_nP_WF$ is also polynomial of degree less than or equal $n$ by the commutativity of the diagram
\[
\xymatrix{
T_nP_WF \ar[r] \ar[d] & \tau_nT_nP_WF \ar[d] \\
P_WT_nP_WF \ar[r] & \tau_nP_WT_nP_WF
}
\]
and the fact that homotopy limits preserve objectwise weak equivalences. It follows that the natural transformation $\eta \colon T_nP_W F \to T_nP_WT_nP_WF$ is a objectwise weak equivalence, as a composite of two objectwise weak equivalences. 

The map $T_nP_W(\eta) \colon T_nP_W F \to T_nP_WT_nP_WF$ is also a objectwise weak equivalence. To see this, note that there is a commutative diagram, 
\[
\xymatrix@C+3ex@R+3ex{
F \ar[d]_{\omega_F} \ar[r]^{\omega_F} \ar@{}[dr]^{(1)} &  P_WF \ar[r]^{\eta_{P_WF}} \ar[d]^{\omega_{P_WF}} \ar@{}[dr]^{(2)} & T_nP_WF \ar[d]^{\omega_{T_nP_WF}} \\
P_WF\ar[r]^{P_W\omega_F} \ar[d]_{\eta_{P_WF}}  \ar@{}[dr]^{(3)} &  P_WP_WF \ar[r]^{P_W\eta_{P_WF}} \ar[d]^{\eta_{P_WP_WF}} \ar@{}[dr]^{(4)} & P_WT_nP_WF \ar[d]^{\eta_{P_WT_nP_WF}}\\
T_nP_WF \ar[r]_{T_nP_W\omega_F} & T_nP_WP_W F \ar[r]_{T_nP_W\eta_{P_WF}} & T_nP_WT_nP_WF 
}
\]

in which, the required map is given by the lower horizontal composite. Since  $P_W$ is a homotopically idempotent functor, $P_W\omega_F$ is a objectwise weak equivalence. It follows that the bottom horizontal map
\[
T_nP_W\omega_F \colon T_nP_WF \longrightarrow T_nP_WP_WF,
\]
of $(3)$ is a weak equivalence sine $T_n$ preserves weak equivalences.

Moreover, $P_W$ being homotopically  idempotent yields that the vertical map
\[
\omega_{P_WF} \colon P_WF \longrightarrow P_WP_WF
\]
in $(2)$ is a objectwise weak equivalence. The right-hand vertical map in this square is also an equivalence by Lemma \ref{lem: fib in null model}. By \cite[Theorem 6.3]{We95}, the top right hand horizontal map
\[
\eta_{P_WF} \colon P_WF \longrightarrow T_nP_WF,
\]
is an approximation of order $n$ in the sense of \cite[Definition 5.16]{We95}. By commutativity of $(2)$, the lower horizontal map 
\[
P_W\eta_{P_WF} \colon P_WP_WF \longrightarrow P_WT_nP_WF,
\]
is an approximation of order $n$. The proof of \cite[Theorem 6.3]{We95} also demonstrates that the vertical maps in $(4)$ are approximations of order $n$, and since three out of the four maps in the lower right square are approximations of order $n$, so too is the lower right hand horizontal map
\[
T_nP_W\eta_{P_WF} \colon T_nP_WP_WF \longrightarrow T_nP_WT_nP_WF.
\]
An application of \cite[Theorem 5.15]{We95} yields that this map is a objectwise weak equivalence as both source and target are polynomial of degree less than or equal $n$. This concludes the proof that the map
\[
T_nP_W(\eta) \colon T_nP_W F \longrightarrow T_nP_WT_nP_WF,
\]
is a objectwise weak equivalence, and verifying \cite[Theorem 9.3(A2)]{Bo01}.

Finally we verify \cite[Theorem 9.3(A3)]{Bo01}. Let
\[
\xymatrix{
A \ar[r]^k \ar[d]_g & B \ar[d]^f \\
C \ar[r]_{h} & D 
}
\]
be a pullback square with $f$ a objectwise fibration between $W$-local $n$-polynomial functors, and $T_nP_Wh \colon T_nP_W C \to T_nP_W D$ a objectwise weak equivalence. By \cite[Theorem 9.9]{Bo01}, we see that the fibre of $k$ is $P_W$-acyclic, i.e. $P_W(\fib(k))$ is objectwise weakly contractible. Since $T_n$ preserves objectwise weak equivalences, we see that $T_nP_W(\fib(k))$ is objectwise weakly contractible, and hence $k$ is a $T_nP_W$-equivalence. 

The fact that the resulting model structure is topological follows from \cite[Theorem 9.1]{Bo01}.  
\end{proof}

This Bousfield-Friedlander localization results in an identical model structure to the $W$-local $n$-polynomial model structure of Proposition \ref{prop: S-local n-poly model structure}

\begin{prop}\label{prop: n-poly for nullification}
For a finite cell complex $W$ there is an equality of model structures
	\[
	{\poly{n}}(\es{J}_0, L_W\T) = \poly{n}(\es{J}_0, P_W\T), 
	\]
	that is, the $W$-local $n$-polynomial model structure and the $W$-periodic $n$-polynomial model structure agree. In particular, these model structures are cellular, proper and topological. 
\end{prop}
\begin{proof}
Both model structures have the same cofibrations, namely the projective cofibrations. It suffices to show that they share the same fibrant objects. Working through the definition of a fibrant object in the Bousfield-Friedlander localization we see that an orthogonal functor $F$ is fibrant if and only if the canonical map $F \to T_nP_WF$ is a objectwise weak equivalence. It follows that $F$ must be $W$-local and $n$-polynomial, hence fibrant in the $W$-local $n$-polynomial model structure. Conversely, if $F$ is fibrant in the $W$-local $n$-polynomial model structure, then the map $F \to P_WF$ is a objectwise weak equivalence and there is a commutative diagram
\[
\xymatrix{
F \ar[r] \ar[d] & P_W F \ar[d] \\
T_nF \ar[r] & T_nP_W F 
}
\]
in which three out of the four arrows are objectwise weak equivalences, hence so to is the right-hand vertical arrow. It follows that $F$ is fibrant in the Bousfield-Friedlander localization.	
\end{proof}

\begin{rem}
The nullification condition here is necessary. The above lemma does not hold in general. To see this, consider the (smashing) localization at the spectrum $E=H\bb{Q}$. The $H\bb{Q}$-local model structure is not right proper, (see Remark \ref{rem: example of not right proper}) yet if this were expressible as a Bousfield-Friedlander localization it would necessarily be right proper, \cite[Theorem 9.3]{Bo01}. 
\end{rem}

\begin{cor}\label{cor: E-local fibs by T_n}
For a finite cell complex $W$ a map $f \colon X \to Y$ is a fibration in the $W$-local $n$-polynomial model structure if and only if $f$ is a fibration in the projective model structure and the square
\[
\xymatrix{
X \ar[r] \ar[d] & T_nP_WX \ar[d] \\
Y \ar[r] & T_nP_WY 
}
\]
is a homotopy pullback square in the projective model structure on $\Fun(\es{J}_0, \T)$.
\end{cor}
%
%

\begin{rem}
It is highly unlikely that this result holds in more general localizations than nullifications. Let $\es{C}$ be a model category and $S$ a set of maps in $\es{C}$ such that the left Bousfield localization of $\es{C}$ at $S$ exists. By \cite[Proposition 3.4.8(1)]{Hi03}  right properness of $\es{C}$ and $L_{S}\es{C}$ is sufficient for a map $f \colon X \to Y$ being a fibration in $L_{S}\es{C}$ if and only if $f$ is a fibration in $\es{C}$ and the square
\[
\xymatrix{
X \ar[r]^{j_X} \ar[d]_{f} & \hat{X} \ar[d]^{\hat{f}} \\
Y \ar[r]_{j_Y}&  \hat{Y}
}
\]
is a homotopy pullback square, where $\hat{f} \colon \hat{X} \to \hat{Y}$ is a $S$-localization of $f$ in the sense of \cite[Definition 3.2.16]{Hi03}. In our situation, Proposition \ref{prop: local model as a nullification} guarantees that a homological localization is right proper if and only if it is a nullification. However, it is not clear in general if right properness of the base model category and the localized model category is a necessary condition for the above description of the fibrations in $L_{S}\es{C}$. 
\end{rem}

		\subsection{Nullifications and homogeneous functors}

In the case of a nullification, the $W$-local $n$-homogeneous model structure of Proposition \ref{prop: E local n-homog model} is not the only way of constructing a model structure with the correct homotopy category. Since the $W$-local model structure on based spaces is right proper, so too is the $W$-local $n$-polynomial model structure and hence we can also follow the more standard procedure and preform a right Bousfield localization at the set 
 \[
\es{K}_n' =  \{ \es{J}_n(U, -) \mid U \in \es{J}\},
 \]
to obtain a local $n$-homogeneous model category structure. 

\begin{prop}\label{prop: W-local n-homog as RBL}
For a finite cell complex $W$ there exists a model structure on the category of orthogonal functors with weak equivalences those maps $X \to Y$ such that 
\[
(T_nP_WX)^{(n)} \longrightarrow (T_nP_W X)^{(n)},
\]
is a objectwise weak equivalence and with fibrations the fibrations of the $W$-local $n$-polynomial model structure. This model structure cellular, proper, stable and topological. We call this the $W$-periodic $n$-homogeneous model structure and denote it $\homog{n}(\es{J}_0, P_W\T)$. 
\end{prop}
\begin{proof}
This is the right Bousfield localization of the $W$-local $n$-polynomial model structure. The proof of which follows exactly as in \cite[Proposition 6.9]{BO13}. Note that this right Bousfield localization exists since the $W$-local $n$-polynomial model structure in right proper and cellular when the localization is a nullification, see Proposition \ref{prop: n-poly for nullification}. 
\end{proof}

This right Bousfield localization behaves like a left Bousfield localization of the $n$-homogeneous model structure in the following sense.

\begin{lem}
For a finite cell complex $W$ the adjoint pair
\[
\adjunction{ \mathds{1}}{\homog{n}(\es{J}_0,\T)}{\homog{n}(\es{J}_0, P_W\T)}{ \mathds{1}},
\]
is a Quillen adjunction.
\end{lem}
\begin{proof}
Since the acyclic cofibrations of the $n$-homogeneous model structure are precisely the acyclic cofibrations of the $n$-polynomial model structure and similarly, the acyclic cofibrations of $W$-periodic $n$-homogeneous model structure are precisely the acyclic cofibrations of the $W$-local $n$-polynomial model structure, the identity functor preserves acyclic cofibrations by Lemma \ref{lem: QA for polynomials}. 

On the other hand, by \cite[Proposition 3.3.16(2)]{Hi03}, cofibrations between cofibrant objects in a right Bousfield localization are cofibrations in the underlying model structure, hence Lemma \ref{lem: QA for polynomials} shows that the identity functor preserves cofibrations between cofibrant objects. The result follows by \cite[Corollary A.2]{DuggerReplacement}.
\end{proof}

An analogous Quillen equivalence is obtained between the $W$-local intermediate category and the $W$-periodic $n$-homogeneous model structure of Proposition \ref{prop: W-local n-homog as RBL} which recall is obtained as a right Bousfield localization of the $W$-local $n$-polynomial model structure. The proof is all but identical to \cite[Theorem 10.1]{BO13}. 

\begin{thm}\label{thm: W-local diff as Quillen}
For a finite cell complex $W$ the adjoint pair
\[
\adjunction{\res_0^n/O(n)}{L_W{\Fun}_{O(n)}(\es{J}_n, O(n)\T)}{\homog{n}(\es{J}_0, P_W\T)}{{\ind}_0^n\varepsilon^*},
\]
is a Quillen equivalence.  
\end{thm}

 Proposition \ref{prop: E local n-homog model} and Proposition \ref{prop: W-local n-homog as RBL} provide two different model structures which both capture the homotopy theory of $W$-locally $n$-homogeneous functors. However, these model structures are not identical. For instance, the $W$-local model structure of Proposition \ref{prop: E local n-homog model} has fibrant objects the $n$-polynomial functors which have $W$-local $n$-th derivative, whereas the fibrant objects of the $W$-periodic $n$-homogeneous model structure (Proposition \ref{prop: W-local n-homog as RBL}) are the $W$-local $n$-polynomial functors. However, they are Quillen equivalent via the identity functor.

\begin{cor}\label{cor: QE for W-local n-homog}
For a finite cell complex $W$ the adjoint pair
\[
\adjunction{\mathds{1}}{\homog{n}(\es{J}_0, L_W\T)}{\homog{n}(\es{J}_0, P_W\T)}{ \mathds{1}},
\]
is a Quillen equivalence.
\end{cor}
\begin{proof}
Since cofibrations between cofibrant objects in $\homog{n}(\es{J}_0, L_W\T)$ are projective cofibrations which are $T_{n}$-equivalences, and the cofibrations between cofibrant objects of $\homog{n}(\es{J}_0, P_W\T)$ are the projective cofibrations, it follows that the identity functor 
\[
 \mathds{1} \colon \homog{n}(\es{J}_0, L_W\T) \longrightarrow \homog{n}(\es{J}_0, P_W\T), 
\]
necessarily preserves cofibrations between cofibrant objects. On the other hand, the identity functor
\[
 \mathds{1} \colon  \homog{n}(\es{J}_0, P_W\T)\longrightarrow \homog{n}(\es{J}_0, L_W\T),
\]
preserves fibrant objects since if $X$ is objectwise $W$-local, $\ind_0^nX$ is objectwise $W$-local, by Lemma \ref{lem: derivative E-local}. It follows that the adjunction is a Quillen adjunction. 
To see that the adjunction is a Quillen equivalence, there is a commutative square 	
\[
\xymatrix@C+2ex@R+1ex{
L_W{\Fun}_{O(n)}(\es{J}_n, O(n)\T) \ar@<1ex>[r]^{\res_0^n/O(n)} \ar@<-1ex>[d]_{\mathds{1}} & \homog{n}(\es{J}_0, L_W\T) \ar@<1ex>[l]^{\ind_0^n\varepsilon^\ast} \ar@<-1ex>[d]_{\mathds{1}} \\
L_W{\Fun}_{O(n)}(\es{J}_n, O(n)\T) \ar@<1ex>[r]^{\res_0^n/O(n)} \ar@<-1ex>[u]_{\mathds{1}} & \homog{n}(\es{J}_0, P_W\T) \ar@<1ex>[l]^{\ind_0^n\varepsilon^\ast}  \ar@<-1ex>[u]_{\mathds{1}}
}
\]
of Quillen adjunctions, in which three-out-of-four are Quillen equivalences by Theorem \ref{thm: E-local diff as Quillen} and Theorem \ref{thm: W-local diff as Quillen}. Hence the remaining Quillen adjunction must also be a Quillen equivalence. 
\end{proof}

It follows that there is a zigzag of Quillen equivalences
\[
\homog{n}(\es{J}_0, P_W\T) \simeq_{Q} \s(L_W\T)[O(n)],
\]
whenever both model structures exist.

	\section{Postnikov sections}

Given a based space $A$, the $k$-th Postnikov section of $A$ is the nullification of $A$ at $S^{k+1}$, i.e., $P_{k}A = P_{S^{k+1}}A$.  Given a diagram of (simplicial, left proper, combinatorial) model categories,  Barwick \cite[Section 5 Application 1]{BarwickLeftRight} and Bergner \cite{BergnerHomotopyLimits} develop a general machinery for producing a model structure which captures the homotopy theory of the homotopy limit of the diagram of model categories. Guti{\'e}rrez and Roitzheim \cite[Section 4]{GR16} applied this to the study of Postnikov sections for model categories, which recovers the classical theory when $\es{C}$ is the Kan-Quillen model structure on simplicial sets. We consider the relationship between Postnikov sections and orthogonal calculus via our local calculus. 

\subsection{A combinatorial model for calculus}

The current theory of homotopy limits of model categories requires that the model categories in question be combinatorial, i.e., locally presentable and cofibrantly generated. Since the category of based compactly generated weak Hausdorff spaces is not locally presentable the Quillen model structure is not combinatorial and hence none of our model categories for orthogonal functors are either. We invite the reader to take for granted that all of our cellular model categories may be replaced by combinatorial model categories by starting with a combinatorial model for the Quillen model structure on based spaces, and hence skip directly to Subsection \ref{ssection: k-types of input}. 

We give the details of these combinatorial replacements here. We replace compactly generated weak Hausdorff spaces with \emph{$\Delta$-generated spaces}; a particular full subcategory of the category of topological spaces, which were developed by Vogt \cite{VogtConvenientHomotopy} and unpublished work of Smith, and are surveyed by Dugger in \cite{DuggerDeltaGeneratedSpaces}. The category of $\Delta$-generated spaces may be equipped with a model structure analogous to the Quillen model structure on compactly generated weak Hausdorff spaces with weak equivalences the weak homotopy equivalences and fibrations the Serre fibrations. This model structure is combinatorial, proper and topological. The existence of the model structure follows from \cite[Subsection 1.9]{DuggerDeltaGeneratedSpaces}. The locally presentable (and hence combinatorial) property follows from \cite[Corollary 3.7]{FRDeltaGeneratedSpaces}. The Quillen equivalence may be extracted from \cite[Subsection 1.9]{DuggerDeltaGeneratedSpaces}.

This combinatorial model for spaces transfers to categories of functors and we obtain a projective model structure on the category of orthogonal functors which is Quillen equivalent to our original projective model structure but is now combinatorial. A left or right Bousfield localization of a combinatorial model category is again combinatorial, hence the $n$-polynomial, $n$-homogeneous and local versions of these model categories are all combinatorial when we begin with the combinatorial model for the projective model structure on orthogonal functors.

\begin{hypothesis}\label{hyp: combinatorial}
	For the remainder of this section, we will assume that all our model structures are combinatorial, since they are all Quillen equivalent to combinatorial model categories using the combinatorial model for based spaces.
\end{hypothesis}

		\subsection{The model structure of $k$-types in orthogonal functors}\label{ssection: k-types of input}

Denote by $I$ the set of generating cofibrations of the projective model structure of orthogonal functors, and denote by $W_k$ the set of maps of the form
\[
B \wedge S^{k+1} \coprod_{A \wedge S^{k+1}} A \wedge D^{k+2} \longrightarrow B \wedge D^{k+2},
\]
where $A \to B$ is a map in $I$. The model category of $k$-types in $\Fun(\es{J}_0, \T)$ is the left Bousfield localization of the projective model structure at $I \Box \{S^{k+1} \to D^{k+2}\}$ used by Guti\'{e}rrez and Roitzheim \cite{GR16} to model Postnikov sections.

\begin{prop}\label{prop: GR k-types in orthogonal functors}
Let $k \geq 0$. Under Hypothesis \ref{hyp: combinatorial}, the model structure of $k$-types in the category of orthogonal functors is identical to the $S^{k+1}$-local model structure, that is, there is an equality of model structures,
\[
P_k{\Fun}(\es{J}_0, \T) := L_{W_k}{\Fun}(\es{J}_0, \T) = {\Fun}(\es{J}_0, L_{S^{k+1}}\T). 
\]	
\end{prop}
\begin{proof}
	It suffices to show that both model structures have the same fibrant objects since the cofibrations in both model structures are identical. To see this, note that by examining the pushout product we can rewrite the set $W_k$ as 
	\[
	W_k = \{ \es{J}_0(U,-) \wedge S^{n+k+1}_+ \longrightarrow \es{J}_0(U,-) \wedge D^{n+k+2}_+ \mid n \geq 0, U \in \es{J}_0\}.
	\]
	It follows by an adjunction argument that an orthogonal functor $Z$ is $W_k$-local if and only if $\pi_i Z(U)$ is trivial for all $i \geq k+1$ and all $U \in \es{J}_0$. This last condition is equivalent to being objectwise $S^{k+1}$-local. 
\end{proof}

		\subsection{The model structure of $k$-types in spectra}

Taking $I_{\s}$ to be the set of generating cofibrations of the stable model structure on $\s$ and denoting again by $W_k$ the relevant pushout product maps, we obtain a similar characterisation of the category of $k$-types in spectra. 

\begin{prop}\label{prop: GR k-types in spectra as stablisation}
Let $k \geq 0$. Under Hypothesis \ref{hyp: combinatorial}, there is an equality of model structures between the model category of $k$-types in spectra, and the stablisation of $S^{k+1}$-local spaces, that is, 
\[
P_k\s := L_{W_k}\s = \s(L_{S^{k+1}}\T).
\]	
\end{prop}
\begin{proof}
Both model structures can be described as particular left Bousfield localizations of the stable model structure on spectra, hence have the same cofibrations. The proof reduces to the fact that the model structures have the same fibrant objects. To see this, note that the fibrant objects of $P_k\s$ are the $k$-truncated $\Omega$-spectra, and the fibrant objects of $\s(L_{S^{k+1}}\T)$ are the levelwise $k$-truncated $\Omega$-spectra. Since both fibrant objects are $\Omega$-spectra a connectivity style argument yields that an $\Omega$-spectrum is $k$-truncated if and only if it is levelwise $k$-truncated, and hence both model structures have the same fibrant objects.
\end{proof}

\begin{rem}
Given a compact Lie group $G$, a similar procedure shows that there is an equality of model structures	
\[
P_k\s[G] := L_{W_k}\s[G] = \s(L_{S^{k+1}}\T)[G].
\]	
\end{rem}

		\subsection{Postnikov reconstruction of orthogonal functors}

The collection of $S^{k+1}$-local model structures on the category of orthogonal functors assembles into a tower of model categories\footnote{A tower of model categories is a special instance of a left Quillen presheaf, that is a diagram of the form $F \colon \es{J}^\op \to \sf{MCat}$ for some small indexing category $\es{J}$.}
\begin{align*}
{\sf{P}}_\bullet \colon \bb{N}^\op &\longrightarrow {\sf{MCat}}, \\
						k    &\longmapsto {\Fun}(\es{J}_0, L_{S^{k+1}}\T),
\end{align*}
where $\sf{MCat}$ denotes the category of model categories and left Quillen functors. The homotopy limit of this tower of model categories recovers the projective model structure on orthogonal functors. The existence of a model structure which captures the homotopy theory of the limit of these model categories follows from \cite[Proposition 2.2]{GR16}. In particular, the homotopy limit model structure is a model structure on the category of sections\footnote{A section $X_\bullet$ of the tower ${\sf{P}}_\bullet$ is a sequence 
\[
\cdots \longrightarrow X_k \longrightarrow X_{k+1} \longrightarrow \cdots \longrightarrow X_0,
\]
of orthogonal functors, and a morphism of sections $f \colon X_\bullet \to Y_\bullet$ is given by maps of orthogonal functors $f_k \colon X_k \to Y_k$ for all $k \geq 0$ subject to a commutative ladder condition.} of the diagram ${\sf{P}}_\bullet$ formed by right Bousfield localizing the injective model structure in which a map of sections is a weak equivalence or cofibration if it is a objectwise weak equivalence or cofibration respectively.

\begin{lem}[{\cite[Theorem 1.3 $\&$ Proposition 2.2]{GR16}}]
There is a combinatorial model structure on the category of sections of ${\sf{P}}_\bullet$ where a map $f_\bullet \colon X_\bullet \to Y_\bullet$ is a fibration if and only if $f_0$ is a fibration in ${\Fun}(\es{J}_0, L_{S^{1}}\T)$ and for every $k \geq 1$ the induced map 
\[
\xymatrix{
X_k \ar@{-->}[dr] \ar@/_3ex/[ddr] \ar@/^3ex/[drr]& & \\
& Y_k \times_{Y_{k-1}}X_{k-1} \ar[r] \ar[d] & X_{k-1} \ar[d] \\
& Y_k \ar[r] & Y_{k-1} \\
}
\]
indicated by a dotted arrow in the above diagram is a fibration in ${\Fun}(\es{J}_0, L_{S^{k+1}}\T)$. A section $X_\bullet$ is cofibrant if and only if $X_n$ is cofibrant in ${\Fun}(\es{J}_0, \T)$ and for every $k \geq 0$, the map $X_{k+1} \to X_k$ is a weak equivalence in ${\Fun}(\es{J}_0, L_{S^{k+1}}\T)$. A map of cofibrant sections is a weak equivalence if and only if the map is a weak equivalence in ${\Fun}(\es{J}_0, L_{S^{k+1}}\T)$ for each $k \geq 0$. We will refer to this model structure as the homotopy limit model structure and denote it by $\holim {\sf{P}}_\bullet$.
\end{lem}

\begin{prop}
Under Hypothesis \ref{hyp: combinatorial} the adjoint pair
\[
\adjunction{\rm{const}}{{\Fun}(\es{J}_0, \T)}{\holim {\sf P}_\bullet}{\lim}
\]
is a Quillen equivalence. 	
\end{prop}
\begin{proof}
The adjoint pair exists, and is a Quillen adjunction by \cite[Lemma 2.4]{GR16}.

To see that the adjoint pair is a Quillen equivalence let $X_\bullet$ be a cofibrant and fibrant section in the homotopy limit model structure. Showing that 
\[
\rm{const}\lim X_\bullet \to X_\bullet,
\]
is a weak equivalence is equivalent to showing that the map $\lim X_\bullet \to X_k$ is a weak equivalence in ${\Fun}(\es{J}_0, L_{S^{k+1}}\T)$ for all $k \geq 0$. This is in turn, equivalent to the map $(\lim X_\bullet)(U) \to X_k(U)$ being a weak equivalence in $L_{S^{k+1}}\T$ for all $k \geq 0$. Since limits in functor categories are computed objectwise, the fact that the unit is a weak equivalence follows from \cite[Theorem 2.5]{GR16}. A similar argument, shows that the counit is also a weak equivalence. 
\end{proof}

\subsection{Postnikov reconstruction for spectra with an $O(n)$-action}
The aim is to show that similar reconstruction theorems may be obtained for the $n$-homogeneous model structures. We first start by investigating analogous theorems for spectra and show that such reconstructions are compatible with the zigzag of Quillen equivalences between spectra with an $O(n)$-action and the $n$-homogeneous model structure. Proposition \ref{prop: GR k-types in spectra as stablisation} and \cite[Subsection 2.1]{GR16} imply that the functor
\begin{align*}
{\sf{P}}_{\bullet}^{\s} \colon \bb{N}^\op &\longrightarrow {\sf{MCat}}, \\
								k &\longmapsto \s(L_{S^{k+1}}\T),
\end{align*}
defines a left Quillen presheaf\footnote{Alternatively, the adjunction $\adjunction{\mathds{1}}{\s(L_{S^{k+2}}\T)}{\s(L_{S^{k+1}}\T)}{\mathds{1}}$, is a Quillen adjunction. This fact follows from the facts that both model structures have the same cofibrations and a $S^{k+1}$-local space is $S^{k+2}$-local as $\langle \Sigma W \rangle \leq  \langle W \rangle$ for all based spaces $W$, see e.g., \cite[\S9.9]{BousfieldPeriodicity}. Hence ${\sf{P}}_{\bullet}^{\s}$ is a left Quillen presheaf.}. This left Quillen presheaf is `convergent' in the following sense.

\begin{prop}\label{prop: limit of spectra}
Under Hypothesis \ref{hyp: combinatorial} the adjoint pair
\[
\adjunction{\rm{const}}{\s}{\holim~{\sf{P}}_{\bullet}^{\s}}{\lim}
\]
is a Quillen equivalence. 	
\end{prop}
\begin{proof}
The fact that the adjoint pair is a Quillen adjunction follows from \cite[Lemma 2.4]{GR16}. 

The left adjoint reflects weak equivalences between cofibrant objects. Indeed, if $X \to Y$ is a map between cofibrant spectra $X$ and $Y$, such that 
\[
\mathrm{const}(X) \longrightarrow \mathrm{const}(Y),
\]
is a weak equivalence in $\holim~{\sf{P}}_{\bullet}^{\s}$, then 
\[
\mathrm{const}(X) \longrightarrow \mathrm{const}(Y),
\]
is a weak equivalence in $\mathrm{Sect}(\bb{N},{\sf{P}}_{\bullet}^{\s})$ by the colocal Whitehead's theorem and the fact that the left adjoint is left Quillen and thus preserves cofibrant objects. It follows that for each $k \in \bb{N}$, the induced map 
\[
\mathrm{const}(X)_k \longrightarrow \mathrm{const}(Y)_k,
\]
is a weak equivalence in $\s(L_{S^{k+1}}\T)$, that is, $X \to Y$ is a weak equivalence in $\s(L_{S^{k+1}}\T)$ for all $k$. Unpacking the definition of a weak equivalence in $\s(L_{S^{k+1}}\T)$ and using the fact that the right adjoint is a right Quillen functor and hence preserves weak equivalences between fibrant objects, we see that the induced map
\[
\lim~P_kX \longrightarrow \lim~P_kY,
\]
is a weak equivalence in $\s$, and hence, so is the map $X \to Y$. 

It is left to show that the derived counit is an isomorphism. Let $Y_\bullet$ be bifibrant in $\holim~{\sf{P}}_{\bullet}^{\s}$. The condition that the counit applied to $Y_\bullet$ is a weak equivalence is equivalent to asking for the map
\[
\lim_{\geq k}~P_kY_\bullet \longrightarrow Y_k,
\]
to be  a weak equivalence in $\s(L_{S^{k+1}}\T)$ for all $k \in \bb{N}$. The structure maps of $Y_\bullet$ induce a map of towers
\[
\xymatrix{
\cdots \ar[r] & Y_j \ar[d] \ar[r] & \cdots \ar[r] & Y_{k+3} \ar[r] \ar[d]  &Y_{k+2} \ar[r] \ar[d] & Y_{k+1} \ar[d] \\
\cdots & Y_{k+1} \ar[r] & \cdots \ar[r] & Y_{k+1} \ar[r] & Y_{k+1} \ar[r] & Y_{k+1}
}
\]
in which each vertical arrow is a weak equivalence in $\s(L_{S^{k+1}}\T)$. This map of towers induces a map
\[
\xymatrix{
0 \ar[r] & \lim^1_{\geq k}~\pi_{i+1}(Y_\bullet) \ar[r] \ar[d] & \pi_i(\lim_{\geq k}~Y_\bullet) \ar[r] \ar[d] & \lim_{\geq k}~\pi_i(Y_\bullet) \ar[r] \ar[d] & 0  \\
0 \ar[r] & \lim^1_{\geq k}~\pi_{i+1}(Y_{k+1}) \ar[r] & \pi_i(\lim_{\geq k}~Y_{k+1}) \ar[r] & \lim_{\geq k}~\pi_i(Y_{k+1}) \ar[r] & 0 
}
\]
of short exact sequences. For $0 \leq i < n$ the left and right hand side maps are isomorphisms hence the map 
\[
\lim_{\geq k}~ Y_\bullet \longrightarrow Y_{k+1},
\]
is a weak equivalence in $\s(L_{S^{k+1}}\T)$ for all $k$, and it follows that the required map 
\[
\lim_{\geq k}~ Y_\bullet \longrightarrow Y_{k+1} \longrightarrow Y_k,
\]
is a weak equivalence in $\s(L_{S^{k+1}}\T)$ for all $k$.
\end{proof}

A similar justification to before provides a left Quillen presheaf
\begin{align*}
{\sf{P}}_\bullet^{\s[O(n)]} \colon \bb{N}^\op &\longrightarrow {\sf{MCat}}, \\
								k &\longmapsto \s(L_{S^{k+1}}\T)[O(n)],
\end{align*}
where $\s(L_{S^{k+1}}\T)[O(n)]$ is the category of $O(n)$-objects in the category of $k$-types in spectra. This is equivalent to the category of $k$-types in spectra with an $O(n)$-action. As a corollary to Proposition \ref{prop: limit of spectra}, we obtain that the induced left Quillen presheaf on spectra with an $O(n)$-action is also suitably convergent. 

\begin{cor}\label{cor: limit of spectra with O(n)-action}
Under Hypothesis \ref{hyp: combinatorial} the adjoint pair
\[
\adjunction{\rm{const}}{\s[O(n)]}{\holim~{\sf{P}}_{\bullet}^{\s[O(n)]}}{\lim}
\]
is a Quillen equivalence. 
\end{cor}

		\subsection{Postnikov reconstruction for the intermediate categories}

The functor
\begin{align*}
{\sf{P}}_\bullet^{\es{J}_n} \colon \bb{N}^{\op} &\longrightarrow {\sf{MCat}}, \\ 
k &\longmapsto L_{S^{k+1}}{\Fun}_{O(n)}(\es{J}_n, O(n)\T),
\end{align*}	
defines a left Quillen presheaf, since there is an equality of model structures between the $S^{k+1}$-local $n$-stable model structure and the model structure of $k$-types in ${\Fun}_{O(n)}(\es{J}_n, O(n)\T)$. The proof of which is completely analogous to the case for spectra, see Proposition \ref{prop: GR k-types in spectra as stablisation}. Since the $S^{k+1}$-local $n$-stable model structure agrees with the model structure of $k$-types, we will denote both model structure by $P_k{\Fun}_{O(n)}(\es{J}_n, O(n)\T)$.

The homotopy limit of this left Quillen presheaf agrees with the homotopy limit of the left Quillen presheaf on spectra with an $O(n)$-action in the sense that the homotopy limit model categories are Quillen equivalent. In detail, the adjunction
\[
\adjunction{(\alpha_n)_!}{{\Fun}_{O(n)}(\es{J}_n, O(n)\T)}{\s[O(n)]}{(\alpha_n)^*},
\]
of \cite[\S8]{BO13} induces an adjunction
\[
\adjunction{(\alpha_n)_!^\bb{N}}{\Fun(\bb{N}, {\Fun}_{O(n)}(\es{J}_n, O(n)\T))}{\Fun(\bb{N}, \s[O(n)])}{(\alpha_n^*)^\bb{N}},
\]
where $(\alpha_n^*)^\bb{N} = (\alpha_n)^* \circ (-)$. This adjunction in turn induces an adjunction
\[
\adjunction{(\alpha_n)_!^\bb{N}}{\holim~{\sf{P}}_\bullet^{\es{J}_n}}{\holim~{\sf{P}}_\bullet^{\s[O(n)]}}{(\alpha_n^*)^\bb{N}}.
\]

\begin{prop}\label{prop: induced adjunction on holim QE}
Under Hypothesis \ref{hyp: combinatorial} the adjoint pair
\[
\adjunction{(\alpha_n)_!^\bb{N}}{\holim~{\sf{P}}_\bullet^{\es{J}_n}}{\holim~{\sf{P}}_\bullet^{\s[O(n)]}}{(\alpha_n^*)^\bb{N}},
\]
is a Quillen equivalence.
\end{prop}
\begin{proof}
Fibrations of the homotopy limit model structure of ${\sf{P}}_\bullet^{\s[O(n)]}$ are precisely the fibrations of the injective model structure on the category of sections of ${\sf{P}}_\bullet^{\s[O(n)]}$ since the homotopy limit model structure is a right Bousfield localization of the injective model structure. A similar characterisation holds for the left Quillen presheaf ${\sf{P}}_\bullet^{\es{J}_n}$, hence to show that the right adjoint preserves fibrations it suffices to show that the left adjoint preserves acyclic cofibrations of the injective model structure on the categories of sections. To see this, note that the adjunction
\[
\adjunction{(\alpha_n)_!}{{\Fun}_{O(n)}(\es{J}_n, O(n)\T)}{\s[O(n)]}{(\alpha_n)^*},
\]
is a Quillen adjunction and hence, so too is the induced adjunction on the injective model structures on the categories of sections. 

To show that the left adjoint preserves cofibrations it suffices to show that cofibrations between cofibrant objects are preserved. As the homotopy limit model structures are right Bousfield localizations \cite[Proposition 3.3.16(2)]{Hi03} implies that cofibrations between cofibrant objects are cofibrations of the injective model structures on the categories of sections which by the analogous reasoning as above are preserved by the left adjoint. This yields that the adjunction in question is a Quillen adjunction. 

To show that the adjunction is a Quillen equivalence notice that the right adjoint reflects weak equivalences between cofibrant objects by the colocal Whitehead's Theorem \cite[Theorem 3.2.13(2)]{Hi03}, and the fact that the induced adjunction on the injective model structures on the categories of sections is a Quillen equivalence since for $B_\bullet \in {\sf{Sect}}(\bb{N}, {\sf{P}}_\bullet^{\es{J}_n})$ and $X_\bullet \in {\sf{Sect}}(\bb{N}, {\sf{P}}_\bullet^{\s[O(n)]})$, a map $B_\bullet \to (\alpha_n^*)^\bb{N}X_\bullet$ is a weak equivalence if and only if for each $k \in \bb{N}$, the map $B_k \to (\alpha_n^*)^\bb{N}X_k$ is a weak equivalence of spectra, which in turn happens if and only if the adjoint map $(\alpha_n)_!B_k \to X_k$ is an $n$-stable equivalence, which is precisely the condition that the adjoint map $(\alpha_n)_!^\bb{N}B_\bullet \to X_\bullet$ is a weak equivalence. 

It is left to show that the derived counit is an isomorphism. Let $Y_\bullet$ be bifibrant in the homotopy limit model structure of the left Quillen presheaf ${\sf{P}}_\bullet^{\s[O(n)]}$. Then the derived counit
\[
(\alpha_n)_!^\bb{N}~\cofrep~((\alpha_n^*)^\bb{N}Y_\bullet) \longrightarrow Y_\bullet,
\]
is a map between cofibrant objects, hence a weak equivalence in the homotopy limit model structure if and only if a weak equivalence in the injective model structure on the category of sections i.e., if and only if for each $k \in \bb{N}$, the induced map
\[
(\alpha_n)_!(\alpha_n)^*Y_k \longrightarrow Y_k,
\]
is a weak equivalence. This last is always a weak equivalence by \cite[Proposition 8.3]{BO13}.
\end{proof}

As a corollary, we see that the left Quillen presheaf ${\sf{P}}_\bullet^{\es{J}_n}$ is convergent.

\begin{cor}\label{cor: intermediate postnikov}
Under hypothesis \ref{hyp: combinatorial} the adjoint pair
\[
\adjunction{\mathrm{const}}{{\Fun}_{O(n)}(\es{J}_n, O(n)\T)}{\holim~{\sf{P}}_\bullet^{\es{J}_n}}{\lim},
\]
is a Quillen equivalence.	
\end{cor}
\begin{proof}
Consider the commutative diagram
\[
\xymatrix@C+3ex{
{\Fun}_{O(n)}(\es{J}_n, O(n)\T) \ar@<1ex>[r]^{(\alpha_n)_!} \ar@<-1ex>[d]_{\rm{const}} & \s[O(n)] \ar@<1ex>[l]^{(\alpha_n)^\ast} \ar@<-1ex>[d]_{\rm{const}} \\
\holim~{\sf{P}}_\bullet^{\es{J}_n} \ar@<-1ex>[u]_{\lim} \ar@<1ex>[r]^{(\alpha_n)_!^\bb{N}}  & \holim~{\sf{P}} \ar@<1ex>[l]^{(\alpha_n^\ast)^\bb{N}} \ar@<-1ex>[u]_{\lim}
}
\]
of Quillen adjunctions in which three out of the four adjoint pairs are Quillen equivalences by \cite[Proposition 8.3]{BO13}, Corollary \ref{cor: limit of spectra with O(n)-action} and Proposition \ref{prop: induced adjunction on holim QE}.  It follows since Quillen equivalences satisfy the $2$-out-of-$3$ property, that the remaining Quillen adjunction is a Quillen equivalence. 
\end{proof}

		\subsection{Postnikov reconstruction for homogeneous functors}

The same approach as we have just employed from moving from spectra with an $O(n)$-action to the intermediate categories yields similar results for the homogeneous model structures. We choose to model $S^{k+1}$-local $n$-homogeneous functors by the $S^{k+1}$-periodic $n$-homogeneous model structures of Proposition \ref{prop: W-local n-homog as RBL}. 

\begin{lem}
The functor
\begin{align*}
{\sf{P}}_{\bullet}^{\homog{n}} \colon \bb{N}^\op &\longrightarrow {\sf{MCat}}, \\
										k &\longrightarrow {\homog{n}}(\es{J}_0, P_{S^{k+1}}\T),
\end{align*}
defines a left Quillen presheaf.
\end{lem}
\begin{proof}
It suffices to show that the adjoint pair
\[
\adjunction{\mathds{1}}{\homog{n}(\es{J}_0, P_{S^{k+2}}\T)}{\homog{n}(\es{J}_0, P_{S^{k+1}}\T)}{\mathds{1}},
\]
is a Quillen adjunction. The adjoint pair
\[
\adjunction{\mathds{1}}{\poly{n}(\es{J}_0, L_{S^{k+2}}\T)}{\poly{n}(\es{J}_0, L_{S^{k+1}}\T)}{\mathds{1}}
\]
is a Quillen adjunction since the composite of Quillen adjunctions is a Quillen adjunction so the adjunction 
\[
\adjunction{\mathds{1}}{{\Fun}(\es{J}_0, L_{S^{k+2}}\T)}{\poly{n}(\es{J}_0, L_{S^{k+1}}\T)}{\mathds{1}}
\]
is a Quillen adjunction, and by \cite[Proposition 3.3.18(1) $\&$ Theorem 3.1.6(1)]{Hi03}, this composite Quillen adjunction extends to the $S^{k+2}$-local $n$-polynomial model structure since $S^{k+1}$-local $n$-polynomial functors are $S^{k+2}$-locally $n$-polynomial. 

An application of \cite[Theorem 3.3.20(2)(a)]{Hi03} yields the desired result about the $n$-homogeneous model structures.  
\end{proof}

Similar proofs to Proposition \ref{prop: induced adjunction on holim QE} and Corollary \ref{cor: intermediate postnikov} yield the following results relating the $n$-homogeneous model structure to the homotopy limit of the tower of $S^{k+1}$-local $n$-homogeneous model structures. 

\begin{prop}
Under Hypothesis \ref{hyp: combinatorial} the adjunction
\[
\adjunction{(\res_0^n/O(n))^\bb{N}}{\holim~{\sf{P}}_\bullet^{\homog{n}}}{\holim~{\sf{P}}_\bullet^{\es{J}_n}}{(\ind_0^n\varepsilon^*)^\bb{N}},
\]
is a Quillen equivalence. 	
\end{prop}

\begin{cor}\label{cor: homog holim}
Under Hypothesis \ref{hyp: combinatorial} the adjunction	
\[
\adjunction{\mathrm{const}}{\homog{n}(\es{J}_0, \T)}{\holim~{\sf{P}}_\bullet^{\homog{n}}}{\lim},
\]	
is a Quillen equivalence.	
\end{cor}


\bibliography{references}
\bibliographystyle{alpha}
\end{document}